\documentclass[11pt,a4paper,reqno]{amsart}
\usepackage{graphicx}
\usepackage[linktocpage=true,colorlinks,citecolor=blue,linkcolor=blue,urlcolor=blue]{hyperref}
\usepackage{amssymb}
\usepackage{cite}
\usepackage{amsmath}
\usepackage{latexsym}
\usepackage{amscd}
\usepackage{tikz}
\usepackage{amsthm}
\usepackage{thm-restate}
\usepackage{mathrsfs}
\usepackage{url}
\usepackage[utf8]{inputenc}
\usepackage[english]{babel}
\usepackage{amsfonts}
\usepackage[textwidth=20mm]{todonotes}

\numberwithin{equation}{section}
\def\R{\mathbb R}
\def\Z{\mathbb Z}
\def\Q{\mathbb Q}
\def\C{\mathbb C}

\def\N{\mathbb N}

\def\CA{\mathcal A}
\def\CC{\mathcal C}

\def\va{\mathbf{a}}
\def\vb{\mathbf{b}}
\def\vx{\mathbf{x}}
\def\vr{\mathbf{r}}
\def\vy{\mathbf{y}}
\def\vN{\mathbf{N}}
\def\vzero{\mathbf{0}}
\def\gcd{\operatorname{gcd}}

\DeclareMathOperator\disc{disc}
\DeclareMathOperator\AP{AP}

\newtheorem{theorem}{Theorem}[section]
\newtheorem{lemma}[theorem]{Lemma}
\newtheorem{proposition}[theorem]{Proposition}
\newtheorem*{theorem*}{Theorem}

\newtheorem{corollary}[theorem]{Corollary}
\newtheorem{conjecture}[theorem]{Conjecture}

\newtheorem{claim}{Claim}

\theoremstyle{remark}
\newtheorem*{remark}{Remark}
\newtheorem*{notations}{Notations}

\theoremstyle{definition}

\theoremstyle{remark}

\numberwithin{equation}{section}

\begin{document}
	\title{Discrepancy of arithmetic progressions in grids}
		\author{Jacob Fox \and Max Wenqiang Xu \and Yunkun Zhou}
	\thanks{Fox is supported by a Packard Fellowship and by NSF award DMS-1855635. \\
		\indent Xu is supported by the Cuthbert C. Hurd Graduate Fellowship in the Mathematical Sciences, Stanford. \\
		\indent Zhou is supported by NSF GRFP Grant DGE-1656518.}
	\address{Department of Mathematics, Stanford University, Stanford, CA, USA}
	\email{\{jacobfox,maxxu,yunkunzhou\}@stanford.edu}

	\maketitle
	
		\begin{abstract}
	We prove that the the discrepancy of arithmetic progressions in the $d$-dimensional grid $\{1, \dots, N\}^d$ is within a constant factor depending only on $d$ of $N^{\frac{d}{2d+2}}$. This extends the case $d=1$, which is a celebrated result of Roth and of Matou\v{s}ek and Spencer, and removes the polylogarithmic factor from the previous upper bound of Valk\'o from about two decades ago. We further prove similarly tight bounds for grids of differing side lengths in many cases.  

	\end{abstract}
	\section{Introduction}
	
	Given a finite set $\Omega$, a coloring of $\Omega$ is a map $\chi: \Omega\to \{1, -1\}$, and  $\chi(A) = \sum_{x\in A}\chi(x).$ For a family $\mathcal{A}$ of subsets of $\Omega$, the discrepancy of $\mathcal{A}$ is defined to be  
\[\disc(\mathcal{A}): = \min_{\chi} \max_{A \in \mathcal{A}} |\chi(A)|,\]
where the minimum is over all colorings of $\Omega$. Let $\mathcal{A}_1$ be the family of arithmetic progressions contained in $[N]:=\{1,\ldots,N\}$. 
	Roth  \cite{Roth} showed using Fourier analysis that there is an absolute constant $c > 0$ such that
	\[\disc(\mathcal{A}_1) \geq cN^{\frac 14}.  \]
In the other direction, Beck \cite{Beck} proved that 
 \[\disc(\mathcal{A}_1) \leq CN^{\frac14}(\log N)^{\frac 52}\]
 for some absolute constant $C$, thereby showing that Roth's lower bound is sharp up to a polylogarithmic factor. Finally, Matou\v{s}ek and Spencer \cite{MS} removed the polylogarithmic factor and resolved this problem of determining the discrepancy up to a constant factor. 
 
It is natural to study the generalization of this problem to higher dimensions. An {\it arithmetic progression in $d$ dimensions} is a set of the form
	\[\AP_d(\va, \vb, l) := \{\va + i \vb: i= 0,1,\dots, l-1\}\]
	where $\va,\vb \in \Z^d$ with $\vb \not = \vzero$, and $l\in \N$. Here $\vb$ is the common difference of the arithmetic progression. Let $\mathcal{A}_d$ be the set of arithmetic progressions in $d$ dimensions that are subsets of $[N]^d$.
	The quantity we are interested in is 
	\[\disc(\mathcal{A}_d): = \min_{\chi: [N]^d \to \{1,-1\} } \max_{A \in \mathcal{A}_d} |\chi(A)|,  \]
	where $\chi(A) = \sum_{x\in A}\chi(x)$. Valk\'o \cite{Valko} proved that there exist constants $c = c(d), C=C(d)$ such that
	\[ cN^{\frac{d}{2d+2}} \leq \text{disc}(\mathcal{A}_d) \leq CN^{\frac{d}{2d+2}}(\log N)^{\frac{5}{2}}. \]

	Valk\'o's proof of the lower bound extends Roth's proof, while the upper bound adapts Beck's proof. The problem of estimating the discrepancy of higher dimensional 
	arithmetic progressions is further discussed in \cite{DSW}.
	In this paper, we remove the polylogarithmic factor in the upper bound and thus determine the quantity up to a constant factor dependent on $d$. 
	
	\begin{theorem}\label{thm:main}For all positive integers $N$ and $d$, 
		we have 
		\[\disc(\mathcal{A}_d) =  \Theta_d(N^{\frac{d}{2d+2}}). \]
	\end{theorem}

The general proof strategy is similar to that in the paper by Matou\v{s}ek and Spencer \cite{MS}. However, new ideas are needed to make the strategy work.  
In particular, we need to overcome some difficulties arising from geometric aspects which requires delicate analysis and tools like Minkowski's theorem and the Lenstra-Lenstra-Lov\'{a}sz basis reduction algorithm.

It is natural to study the generalization of the problem to the discrepancy of arithmetic progressions in grids of side lengths that are not necessarily equal. Given positive integers $N_1, \dots, N_d$, let $\Omega = [N_1]\times \cdots \times [N_d] \subseteq \Z^d$ and $\mathcal{A}_\vN$ be the set of arithmetic progressions in $d$ dimensions that are subsets of $\Omega$, where $\vN = (N_1, \dots, N_d)$. The discrepancy is defined in a similar way, 
	\[\disc(\mathcal{A}_\vN):= \min_{\chi: \Omega\to \{ 1, -1\}}\max_{A\in \mathcal{A}_\vN} |\chi(A)|.\] 
In the proof of Theorem \ref{thm:main}, we shall see that we will have to consider more generally grids with side lengths of comparable size (see \eqref{eqn:almost-cubes}). 
\begin{theorem}\label{thm:almost-cubes}
	    For any positive integer $d$ and positive integers $N_1, \dots, N_d$, if $\delta > 0$ satisfies that
	    \begin{equation}\label{eqn:almost-cubes}
	    N_1\cdots N_d \leq  \left(\min_{1\leq i\leq d}N_i\right)^{d+1-\delta}, 
	    \end{equation}
	    	    then there exist positive constants $c_d, C_d$ such that for $\vN = (N_1, \dots, N_d)$,
	    \[c_d\left(N_1\cdots N_d\right)^\frac{1}{2d+2} \le \disc(\mathcal{A}_{\vN}) 
	     \leq C_d\cdot \frac{1}{\delta}\left(N_1\cdots N_d\right)^\frac{1}{2d+2}.
	    \]
\end{theorem}
We remark that Theorem~\ref{thm:almost-cubes} implies Theorem~\ref{thm:main} by choosing $\delta =1$ and $N_1 = N_2 = \cdots = N_d = N$. 

The lower bound in Theorem~\ref{thm:almost-cubes} holds even without condition \eqref{eqn:almost-cubes}. Theorem~\ref{thm:rectangular} gives a more general lower bound, and a matching upper bound up to a sub-logarithmic factor for general grids of differing side lengths. The lower bound in Theorem~\ref{thm:rectangular} implies the lower bound in Theorem~\ref{thm:almost-cubes} by taking $I = [d]$ in the maximum.
The proof of lower bound uses Fourier analytic tools. 

\begin{theorem}\label{thm:rectangular}
	    For any positive integer $d$ and $\vN = (N_1, \dots, N_d)$ where the $N_i$'s are positive integers whose product is at least three, there exist positive constants $c_d, C_d$ such that
	    \begin{equation}\label{eqn: general rectangular}
	        c_d\max_{I\subseteq [d]}\left(\prod_{i\in I}N_i\right)^\frac{1}{2|I|+2} \le \disc(\mathcal{A}_{\vN}) 
	     \leq C_d \frac{\log (N_1\cdots N_d)}{\log \log (N_1\cdots N_d)} \cdot \max_{I\subseteq [d]}\left(\prod_{i\in I}N_i\right)^{\frac{1}{2|I|+2}}.
	    \end{equation}

	    Here by convention if $I = \emptyset$ then $\prod_{i\in I}N_i = 1$.
\end{theorem} 
\begin{remark}
Since $\disc(\mathcal A_\vN)$ does not depend on the order of the $N_i$'s, we may assume without loss of generality that $N_1\geq N_2\geq \cdots \geq N_d$. In this case \[\max_{I\subseteq [d]}\left(\prod_{i\in I}N_i\right)^{\frac{1}{2|I|+2}} = \max_{1\leq k\leq d}\left(\prod_{i=1}^k N_i\right)^\frac{1}{2k+2}.\]
\end{remark}

\noindent {\bf Organization.} 
In Section~\ref{sec:decomposition}, we show how to efficiently decompose arithmetic progressions into ``canonical'' arithmetic progressions and provide an upper bound on the number of such canonical arithmetic progressions with a given size. This implies that a coloring which has low discrepancy in canonical arithmetic progressions of each possible size also achieves low discrepancy for all arithmetic progressions (see Lemma \ref{lem:decomposition} for details). We prove the upper bound in Theorem~\ref{thm:almost-cubes} in Section~\ref{sec: first coloring} by showing the existence of a coloring which has low discrepancy in canonical arithmetic progressions of each possible size. In the proof we use an improved bound on the number of canonical arithmetic progressions with each given size, the proof of which is deferred to Section~\ref{sec:geometry}. We finally study the case of grids with different side lengths in the last three sections. We prove the lower bound of Theorem~\ref{thm:rectangular} in Section~\ref{sec: rec-lower} and the upper bound of Theorem~\ref{thm:rectangular} in Section~\ref{sec: rec-upper}. Finally, we have some concluding remarks in Section~\ref{sec: conclusion and open problem}, including a conjecture that the lower bound for the 
the discrepancy for grids in Theorem \ref{thm:rectangular} is tight up to the constant factor.
	
	\begin{notations}
	Throughout the paper, all logarithms are base $e$ unless specified.
    We generally assume that $d$ is fixed, except in Section~\ref{sec:geometry} where the proof relies on an induction argument on $d$.
	We use symbols $c, c_1, c_2, C_0, C, C^{*}$ to denote positive absolute constants, and $c_d, C_d$ to denote those that only depend on $d$.
	We use notation $f = O_d(g)$ if there exists a positive constant $C_d$ so that $f \leq C_dg$. 
	\end{notations}
	
	\section{Decomposition}\label{sec:decomposition}
	Let $N_1, \dots, N_d$ be positive integers, $\Omega = [N_1]\times \cdots \times [N_d]$, and $\vN = (N_1, \dots, N_d)$.
	
	To find a coloring giving low discrepancy,
	the general idea is to apply a partial coloring lemma (specifically
	Lemma \ref{lem:partial-coloring-lemma}, whose proof uses the entropy method)
	to repeatedly partially color $\Omega$ until we get a full coloring of $\Omega$ with low discrepancy. At each stage, we color a constant fraction of the remaining uncolored elements, 
	until we get a full coloring of $\Omega$ with low discrepancy. To accomplish this, we show that for any $X\subseteq \Omega$, there is a partial coloring of $X$ with low discrepancy. Once we have this statement, we may apply this with $X$ as $\Omega$ in the initial iteration to get a partial coloring of $\Omega$, and then pick $X$ as the set of uncolored elements of $\Omega$ in later iterations. Hence more generally the set family we need to consider is $(X, \CA_X)$ where $\CA_X := \{A\cap X: A\in \CA_\vN\}.$
	
The family of sets $\CA_X$ is too large if we want to apply 
Lemma \ref{lem:partial-coloring-lemma}
on $(X, \CA_X)$ directly. Instead we apply it to a small subfamily $\CC_X \subseteq \CA_X$ so that any set in $\CA_X$ can be efficiently decomposed into sets in $\CC_X$.
	For each $\vb\in \Z^d\setminus \{\vzero\}$, we may partition elements in $X$ into congruence classes modulo $\vb$. For each congruence class $I = \{\vx\in X: \vx \equiv \va \pmod \vb\}$, since distinct elements in $I$ differ by nonzero multiples of $\vb$, and their dot products with $\vb$ differ by nonzero multiples of $\|\vb\|^2\ne 0$, we may order elements in $I$ by their dot products with $\vb$. Write $I = \{\vx_1, \vx_2, \dots, \vx_l\}$, where the subscripts respect the ordering and $l = |I|$. Now any set in $\CA_X$ can be written as $\{\vx_u: i\leq u \leq j\}$ for some $(\vb, I, i, j)$.
	We use the following decomposition. For each $\vb\in \Z^d\setminus \{\vzero\}$ and congruence class $I$ modulo $\vb$, we consider sets of the form $\{\vx_u: (j-1)s + 1 \leq u \leq js\}$ where $s = 2^t$ is a power of $2$, and $1\leq j\leq \lfloor l/s\rfloor$. Let $\CC_X$ be the collection of such sets for all $(\vb, I)$. All sets in $\CC_X$ are of sizes powers of $2$.
	\begin{lemma}\label{lem:decomposition}
	Let $b: \N \to (0, \infty)$ be a function. If $\chi$ is a partial coloring of $X$ so that
	\[|\chi(S)| \leq b(|S|)\]
	for all $S\in \CC_X$, then
	\[|\chi(A)| \leq 2\sum_{s: s = 2^t}b(s)\]
	for all $A\in \CA_X$.
	\end{lemma}
	\begin{proof}
	For any $A\in \CA_X$, we know that it can be written as $A_0 \cap X$ for some arithmetic progression $A_0\in \CA_\vN$. Let $\vb$ be the common difference of $A_0$, and let $I$ be intersection of $X$ and the congruence class mod $\vb$ containing $A_0$. Then $A$ is a subset of $I$. Moreover, as we describe in the procedure above, if we order elements in $I = \{\vx_1, \vx_2, \dots, \vx_l\}$ by their dot product with $\vb$, we know that $A$ must be in the form $\{\vx_u: i\leq u \leq j\}$. We write $A = A_1\setminus A_2$ where $A_1 = \{\vx_u: 1\leq u \leq j\}$ and $A_2 = \{\vx_u: 1\leq u \leq i-1\}$. Then we know that $A_1$ can be written as a disjoint union of sets in $\CC_X$ of different sizes $A_1 = S_1 \cup S_2 \cup \dots \cup S_t$ using the binary representation of $j$, where $t$ is the number of digits $1$ in the representation. Also note that all sets in $\CC_X$ are of sizes powers of $2$, so we have
	\[|\chi(A_1)| = \left|\sum_{k=1}^t \chi(S_k)\right| \leq \sum_{k=1}^t |\chi(S_k)| \leq \sum_{s: s = 2^t}b(s).\]
	We may prove the similar inequality for $\chi(A_2)$ by replacing $j$ with $i-1$. Combining them we get
	\[|\chi(A)| = |\chi(A_1) - \chi(A_2)| \leq 2\sum_{s: s = 2^t}b(s).\]
	\end{proof}
	

To apply the partial coloring lemma, Lemma \ref{lem:partial-coloring-lemma}, to $(X, \CC_X)$, we need to estimate the number of sets of each size, and pick each $\Delta_S$ appropriately. Let $s = 2^t$ be any power of $2$, we define $f(s, X)$ to be the number of sets of size $s$ in $\CC_X$. Note that $f(1, X) = |X|$. 
	
	For a positive integer $s$, a finite set $X \subseteq \Z^d$ and a vector $\vb\in \Z^d\setminus \{\vzero\}$, let $U^d(X, \vb, s)$ denote the set of all $x\in X$, whose residue class mod $\vb$ contains at least $s$ elements in $X$, or formally $\{x'\in X: x'\equiv x \pmod \vb\}$ is of size at least $s$. The following inequality shows how these sets $U^d$ relate to the quantity $f(s, X)$.
	
	\begin{lemma}\label{lem:counting-subsets-by-U} Let $X\subseteq \Omega$, and $s$ be a power of $2$. Then
	\begin{equation}\label{eqn:counting-subsets-by-U}f(s, X) \leq \frac{1}{s}\sum_{\vb\in \Z^d\setminus \{\vzero\}}|U^d(X, \vb, s)|.\end{equation}
	\end{lemma}
	\begin{proof}We would like to estimate the number of sets in $\CC_X$ of size $s$. For each $\vb\in \Z^d\setminus \{\vzero\}$, we partition $X$ into the congruence classes modulo $\vb$ which we denote by $I_1, \dots, I_t$. By definition, $U^d(X, \vb, s)$ is the disjoint union of all $I_k$ that contains at least $s$ elements.
	Each set of size $s$ in $\CC_X$ lies entirely in some $I_k$ for some appropriate choice of $\vb$ and $I_k$. The number of such sets in $I_k$ 
	is at most $\lfloor |I_k|/s\rfloor$.
	Therefore if we sum over all congruence classes, the number of sets in $\CC_X$ of size $s$ for any fixed $\vb$ is
	\[\sum_{k=1}^t \lfloor |I_k|/s\rfloor \leq \sum_{k: |I_k| \geq s} \frac{|I_k|}{s} = \frac{|U^d(X, \vb, s)|}{s}.\]
	Summing over all possible common differences $\vb\in \Z^d\setminus \{\vzero\}$, we obtain \eqref{eqn:counting-subsets-by-U}. \end{proof}
	
	Hence we need to estimate the sum of the numbers of elements in these $U^d$ sets. We have the following simple upper bound.
	\begin{lemma}\label{lem:counting-U-1}
	For any $s\geq 2$ and $X\subseteq [N_1]\times \cdots \times [N_d]$, we have
	\[\sum_{\vb\in \Z^d\setminus \{\vzero\}}|U^d(X, \vb, s)|\leq |X| \cdot \prod_{i=1}^d \left(4\frac{N_i}{s}+1\right).\]
	\end{lemma}
	\begin{proof}
	    We focus on those $\vb$ for which $U^d(X, \vb, s)$ is nonempty. If $U^d(X, \vb, s)$ is nonempty, then we know that for each $i$, the $i$-th coordinate of $\vb$ is in the interval $(-\frac{N_i}{s-1}, \frac{N_i}{s-1})$. Therefore, the number of nonempty $U^d(X, \vb, s)$ is at most
	\[\prod_{i=1}^d \left(\frac{2N_i}{s-1}+1\right) \leq \prod_{i=1}^d \left(4\frac{N_i}{s}+1\right).\]
	Applying the trivial bound $|U^d(X, \vb, s)| \leq |X| = m$, we get the desired inequality.
	\end{proof}
	
By combining Lemma~\ref{lem:counting-subsets-by-U} and Lemma~\ref{lem:counting-U-1} above, the following upper bound on $f(s, X)$ holds.
	\begin{corollary}\label{cor:counting-subsets-1}
	For any $1\leq s \leq \min_{1\leq i\leq d}N_i$, we have
	\[f(s, X) \leq 5^d\frac{N_1\cdots N_d}{s^{d+1}}|X|.\]
	\end{corollary}
	\begin{proof}
	When $s = 1$, we have $f(1, X) = |X|$ and  the inequality clearly holds. In the remaining cases, $2\leq s \leq \min_{1\leq i\leq d}N_i$, we would like to apply Lemma~\ref{lem:counting-subsets-by-U} and Lemma~\ref{lem:counting-U-1}. As $s\leq N_i$ for each $1\leq i\leq d$, it follows that 
	\[f(s, X) \leq \frac{1}{s} \cdot |X|\prod_{i=1}^d \left(4\frac{N_i}{s}+1\right) \leq 5^d\frac{N_1\cdots N_d}{s^{d+1}}|X|.\]\end{proof}
We remark that in Section~\ref{sec:geometry} we prove Lemma \ref{lem:counting-subsets-3} which, together with Lemma \ref{lem:counting-subsets-by-U}, gives a better upper bound on $f(s,X)$ than Corollary \ref{cor:counting-subsets-1}  for a certain range of $s$. 

	\section{Proof of the upper bound in Theorem \ref{thm:almost-cubes}}\label{sec: first coloring}
	
		We use the following version of the partial coloring lemma which was first proved by Matou\v{s}ek and Spencer in \cite{MS}, to show the existence of a partial coloring that colors a constant fraction of elements of a set system with low discrepancy.
\begin{lemma}[Section 4.6 in \cite{Mat}]\label{lem:partial-coloring-lemma}
	Let $(\Omega, \CC)$ be a set system on $n$ elements, and let a number $\Delta_S > 0$ be given for each set $S\in \CC$. Suppose that
	\begin{equation}\label{eqn:partial-coloring-lemma} 
	    \sum_{S\in \CC:S\ne \emptyset}g\left(\frac{\Delta_S}{\sqrt{|S|}}\right) \leq \frac{n}{5}
	\end{equation}
	where
	\begin{equation}\label{eqn:def-g}
	    g(\lambda) = \begin{cases} 10e^{-\lambda^2/4} & \textrm{if }\lambda \geq 2, \\ 10\log(1 + 2\lambda^{-1}) & \textrm{if } 0 < \lambda < 2. \end{cases}
	\end{equation}
	Then there exists a partial coloring $\chi$ that assigns $\pm1$ to at least $n/10$ variables (and $0$ to the rest), satisfying $|\chi(S)| \leq \Delta_S$ for each $S\in \CC$.
	\end{lemma}

	We apply Lemma~\ref{lem:partial-coloring-lemma} to the set system $(X, \CC_X)$ defined in Section~\ref{sec:decomposition}. In \eqref{eqn:def-g} we intentionally choose $g$ to be monotonically decreasing. We further show the following property of $g$.
	\begin{lemma}\label{lem:choice-of-b}
	Let $d\in \N$ and $c = 10d+2400$. Let $K$ be a positive real number, and let $b: \N\to (0, \infty)$ be defined as
	\begin{equation}\label{eqn:def-b-as-K}
	b(s) = \begin{cases}c\sqrt{s} \cdot \left(sK^{-\frac{1}{d+1}}\right)^{-1} &\mbox{if } s \geq K^{\frac{1}{d+1}} \\ c\sqrt{s} \cdot \left(sK^{-\frac{1}{d+1}}\right)^{-0.1} &\mbox{if } s  <  K^{\frac{1}{d+1}}\end{cases}.\end{equation}
	Then for $g$ as defined in \eqref{eqn:def-g}, we have
	\[\sum_{i = 0}^\infty \frac{K}{2^{i(d+1)}}g\left(\frac{b(2^i)}{2^{i/2}}\right) \leq 1.\]
	\end{lemma}
	\begin{proof}
	Let $s_i = 2^i$ and $\tau_i = 2^iK^{-\frac{1}{d+1}}$, where $i$ takes nonnegative integer values. Now we may rewrite 
	\[\frac{b(2^i)}{2^{i/2}} = \begin{cases}c\tau_i^{-1} &\mbox{if } \tau_i \geq 1 \\ c\tau_i^{-0.1} &\mbox{if } \tau_i < 1\end{cases}.\]
	Since the right hand side only depends on $\tau_i$, we denote it as $\lambda(\tau_i)$. Therefore we have
	\begin{equation}\label{eqn:rewrite-sum-as-tau}\sum_{i = 0}^\infty \frac{K}{2^{i(d+1)}}g\left(\frac{b(2^i)}{2^{i/2}}\right) = \sum_{i=0}^\infty \tau_i^{-d-1}g(\lambda(\tau_i)) = \sum_{\tau_i < 1}\tau_i^{-d-1}g(c\tau_i^{-0.1}) + \sum_{\tau_i \geq 1}\tau_i^{-d-1}g(c\tau_i^{-1}).\end{equation}
	We bound two terms on the right hand side of \eqref{eqn:rewrite-sum-as-tau} separately. For the first term, note that $\{\tau_i < 1\}$ can be seen as a geometric sequence with ratio $1/2$. For any $x > 0$ and positive integer $t$, we have $e^x \geq \frac{x^t}{t!}$. Thus by setting $x = \tau_i^{-0.2}$ and $t = 5d+10$ we get that the series
	\[\sum_{\tau_i < 1}\tau_i^{-d-1}e^{-\tau_i^{-0.2}} \leq \sum_{\tau_i < 1}\tau_i^{-d-1}\cdot (5d+10)!\tau_i^{0.2\cdot (5d+10)} = (5d+10)!\sum_{\tau_i<1}\tau_i \leq 2(5d+10)!.\]
	We denote $c_1 = 2(5d+10)!$. Note that $c = 10d+2400$ satisfies $c^2/4 > 1 + \log (20c_1)$. For $\tau < 1$, we have $c\tau^{-0.1} > 2$, so $g(c\tau^{-0.1})$ uses the branch $g(\lambda) = 10e^{-\lambda^2/4}$. Therefore
	\[g(c\tau^{-0.1}) = 10e^{-\frac{c^2\tau^{-0.2}}{4}} \leq 10e^{-(1+\log (20c_1))\tau_i^{-0.2}} \leq 10e^{-\tau^{-0.2} - \log 20c_1} = \frac{1}{2c_1}e^{-\tau^{-0.2}}.\]
	Using this, we may bound the first term on the right hand side of \eqref{eqn:rewrite-sum-as-tau} as following
	\begin{equation}\label{eqn:tau-first-term}
	\sum_{\tau_i < 1}\tau_i^{-d-1}g(c\tau_i^{-0.1}) \leq \sum_{\tau_i < 1} \tau_i^{-d-1}\cdot \frac{1}{2c_1}e^{-\tau_i^{-0.2}} \leq \frac{1}{2c_1}\cdot c_1 = \frac{1}{2}.
	\end{equation}
	Now we bound the second term on the right hand side of \eqref{eqn:rewrite-sum-as-tau}. As $\{\tau_i \geq 1\}$ forms a geometric sequence with ratio $2$ and $\log(1+2\tau_i)\leq 3\tau_i$ when $\tau_i \geq T > 1$, we have for $T = 240$, 
	\[\sum_{\tau_i > T} \tau_i^{-d-1}\log(1+2\tau_i) \leq \sum_{\tau_i > T}\tau_i^{-d-1} \cdot 3\tau_i = 3\sum_{\tau_i > T}\tau_i^{-d} \leq 3\sum_{\tau_i > T}\tau_i^{-1} \leq \frac{6}{T} = \frac{1}{40}.\]
For any $c = c(d) > 1$ and $\tau\geq T$, as $g$ is monotonically decreasing, we have $g(c\tau^{-1}) \leq g(\tau^{-1}) = 10\log(1+2\tau)$ where we use the branch $g(\lambda) = 10\log(1+2\lambda^{-1})$ as $\tau^{-1} < 2$. Hence
	\begin{equation}\label{eqn:tau-second-term-second-half}\sum_{\tau_i \geq T}\tau_i^{-d-1}g(c\tau_i^{-1}) \leq \sum_{\tau_i \geq T}\tau_i^{-d-1} \cdot 10\log(1+2\tau_i) \leq 10\cdot \frac{1}{40} = \frac{1}{4}.\end{equation}
	Finally, for $\tau_i < T$, we have $g(c\tau^{-1}) \leq g(cT^{-1})$. As $g(10)\leq \frac{1}{8}$ and $c > 2400 = 10T$, $g(cT^{-1}) \leq \frac{1}{8}$. Consequently,
	\begin{equation}\label{eqn:tau-second-term-first-half}
	\sum_{1\leq \tau_i < T}\tau_i^{-d-1}g(c\tau_i^{-1}) \leq \sum_{1\leq \tau_i < T}\tau_i^{-d-1}\cdot \frac{1}{8} < \frac{1}{4}.
	\end{equation}
	Here in the last inequality we use that $\{\tau_i: 1\leq \tau_i < T\}$ forms a geometric series with ratio $2$ and the initial term is at least $1$.

Substituting the bounds \eqref{eqn:tau-first-term}, \eqref{eqn:tau-second-term-second-half}, and \eqref{eqn:tau-second-term-first-half} into \eqref{eqn:rewrite-sum-as-tau}, we obtain the desired inequality. 

	\end{proof}

	We choose the function $b$ of the form specified in Lemma~\ref{lem:choice-of-b} as it has a good summation property over powers of $2$. This is illustrated by the following lemma.
	\begin{lemma}\label{lem:sum-b}
	Let $d$ be a positive integer, and $c, K$ be positive real numbers. Let $b: \N\to (0, \infty)$ be the function defined in \eqref{eqn:def-b-as-K}. Suppose that $u < v$ are positive real numbers. Then
	\[\sum_{s: s = 2^t \in [u, v]}b(s) \leq 5cK^\frac{1}{2d+2}\min\left(\left(vK^{-\frac{1}{d+1}}\right)^{0.4}, \left(uK^{-\frac{1}{d+1}}\right)^{-0.5}\right).\]
	\end{lemma}
	\begin{proof}
Here we assume that $s$ takes value over powers of $2$. Note that
	\[b(s) = \min\left(c\sqrt{s}\cdot \left(sK^{-\frac{1}{d+1}}\right)^{-1}, c\sqrt{s}\cdot \left(sK^{-\frac{1}{d+1}}\right)^{-0.1} \right).\]
	Using that $b(s) \leq c\sqrt{s}\cdot \left(sK^{-\frac{1}{d+1}}\right)^{-0.1}$, we get
	\[\sum_{u\leq s \leq v} b(s) \leq cK^{\frac{1}{10(d+1)}} \sum_{u\leq s \leq v}s^{0.4} \leq cK^{\frac{0.1}{d+1}}v^{0.4}\left(\sum_{j \leq 0}2^{0.4j}\right) \leq 5cK^\frac{1}{2d+2}\left(vK^{-\frac{1}{d+1}}\right)^{0.4}.\]
	Similarly using $b(s) \leq c\sqrt{s}\cdot \left(sK^{-\frac{1}{d+1}}\right)^{-1}$, we have
	\[\sum_{u\leq s \leq v} b(s) \leq cK^{\frac{1}{d+1}} \sum_{u\leq s \leq v}s^{-0.5} \leq cK^{\frac{1}{d+1}}u^{-0.5}\left(\sum_{j \geq 0}2^{-0.5j}\right) \leq 5cK^\frac{1}{2d+2}\left(uK^{-\frac{1}{d+1}}\right)^{-0.5}.\]
	\end{proof}
	
	Combining the lemmas above with Corollary~\ref{cor:counting-subsets-1}, we derive below  Proposition~\ref{prop:first-partial-coloring}, which is slightly weaker than what we need to prove Theorem~\ref{thm:main}. Indeed, if we iteratively apply Proposition~\ref{prop:first-partial-coloring} to the remaining uncolored elements, we get $\disc(\CA_d)=O_d(N^{\frac{d}{2d+2}}\log N)$. 
	\begin{proposition}\label{prop:first-partial-coloring}
	Let $d\in \N$ and $c = 500d+10^6$. For any $X\subseteq [N]^d$, there exists a partial coloring $\chi: X \to \{-1, 0, 1\}$ that assigns $\pm1$ to at least $|X|/10$ elements in $X$ such that
	\[\max_{A_0\in \CA_d} |\chi(A_0\cap X)| \leq cN^{\frac{d}{2d+2}}.\]
	\end{proposition}
	\begin{proof}
	Here we treat $d$ as a constant. Suppose that $|X| = m$. Let $b: \N\to (0, \infty)$ be determined later. We want to apply Lemma \ref{lem:partial-coloring-lemma} to the set system $(X, \CC_X)$ to find a partial coloring $\chi: X\to \{0, \pm 1\}$ that assigns $\pm 1$ to at least $m/10$ elements, and that $|\chi(S)| \leq b(|S|)$ for any $S\in \CC_X$. By Lemma~\ref{lem:partial-coloring-lemma}, it suffices to ensure that $b$ satisfies the inequality
	\begin{equation}\label{eqn:prop-first-target}
	\sum_{s: s = 2^t \leq N} f(s, X) g\left(\frac{b(s)}{\sqrt{s}}\right) \leq m/5.\end{equation}
	By Corollary~\ref{cor:counting-subsets-1}, we know that $f(s, X) \leq 5^d \frac{N^dm}{s^{d+1}}.$ It now suffices to show
	\begin{equation}\label{eqn:prop-first-target-converted}
	\sum_{s: s = 2^t \leq N}  5^d \frac{N^dm}{s^{d+1}} g\left(\frac{b(s)}{\sqrt{s}}\right) \leq m/5.
	\end{equation}
	Let $K = 5^{d+1}N^d$. By Lemma~\ref{lem:choice-of-b}, if we set $b$ as in \eqref{eqn:def-b-as-K} with $c_1 = 10d+2400$
	\begin{equation}\label{eqn:prop-first-def-b}
	b(s) = \begin{cases}c_1\sqrt{s} \cdot \left(sK^{-\frac{1}{d+1}}\right)^{-1} &\mbox{if } s \geq K^{\frac{1}{d+1}} \\ c_1\sqrt{s} \cdot \left(sK^{-\frac{1}{d+1}}\right)^{-0.1} &\mbox{if } s  <  K^{\frac{1}{d+1}}\end{cases},
	\end{equation}
	then \eqref{eqn:prop-first-target-converted} is satisfied. Therefore by Lemma~\ref{lem:partial-coloring-lemma}, we know that there exists a partial coloring $\chi$ that assigns $\pm 1$ to at least $m/10$ elements, and that $|\chi(S)| \leq b(|S|)$ for any $S\in \CC_X$. By Lemma~\ref{lem:decomposition}, we know that for any $A\in \CA_X$, $|\chi(A)| \leq 2\sum_{s:s=2^t}b(s)$. Now we have
	\[2\sum_{s: s=2^t \leq N}b(s) \leq 2\sum_{s: s=2^t \in [1, K^{\frac{1}{d+1}}]} b(s) +  2\sum_{s: s=2^t \in [K^{\frac{1}{d+1}}, N+1]} b(s) \leq 20c_1K^\frac{1}{2d+2}.\]
In the second inequality above we use Lemma~\ref{lem:sum-b} for $(u, v) = (1, K^{\frac{1}{d+1}})$ and $(u, v) = (K^{\frac{1}{d+1}}, N)$ respectively. Note that $K^\frac{1}{2d+2} = 5^{\frac{d+1}{2d+2}}N^{\frac{d}{2d+2}}=\sqrt{5}N^{\frac{d}{2d+2}}$. Hence we conclude that we can always find partial coloring $\chi$ that assigns $\pm 1$ to at least $m/10$ elements in $X$, and satisfies
	\[\max_{A_0\in \CA_d} |\chi(A_0\cap X)| = \max_{A\in \CA_X} |\chi(A)| \leq cN^{\frac{d}{2d+2}},\]
	where $c = 20\sqrt{5}c_1 < 500d + 10^6$.
	\end{proof}
	
	Not surprisingly, to improve on the bound above, we need to improve on Corollary~\ref{cor:counting-subsets-1}. In particular, we will show in Section~\ref{sec:geometry} that the following holds.
	
	\begin{restatable}{lemma}{counting}\label{lem:counting-subsets-3}
	There exists an absolute constant $C_0$ such that the following holds. Let $d$ be a positive integer. Given positive integers $N_1, N_2, \dots, N_d$ satisfying $N_1\cdots N_d \leq (\min_{1\leq i\leq d} N_i)^{d+1-\delta}$ for some $\delta\in (0, 1]$, suppose that $X\subseteq [N_1]\times \cdots \times [N_d]$ is of size $m$. Letting $\rho = \frac{m}{N_1\cdots N_d}$, if integer $s$ satisfies $(N_1\cdots N_d)^{\frac{1}{d+1}}\rho^\frac{\delta}{4^d(d+1)} \leq s \leq (\min_{1\leq i\leq d} N_i)\rho^\beta$ for some $\beta \in (0, 1/2)$, then
	\begin{equation}\label{eqn:counting-subsets-3}\sum_{\vb\in \Z^d\setminus \{\vzero\}} |U^d(X, \vb, s)| \leq C_02^{d^3}5^d\frac{mN_1N_2\cdots N_d}{s^d} \cdot \rho^{\frac{\min(\beta, \delta)}{4^d\cdot(d+2)!}}.\end{equation}
	\end{restatable}
	
	Assuming Lemma~\ref{lem:counting-subsets-3}, we prove the following improvement on Proposition~\ref{prop:first-partial-coloring}.
	\begin{proposition}\label{prop:partial-coloring-almost-cube}
	Let $d$ be a positive integer. There exist constants $C_d$ and $c_d$ such that the following holds. Let $N_1, N_2, \dots, N_d$ be positive integers satisfying $N_1\cdots N_d \leq (\min_{1\leq i\leq d} N_i)^{d+1-\delta}$ for some $\delta\in (0, 1]$. For any nonempty $X\subseteq [N_1]\times \cdots \times [N_d]$, there exists a partial coloring $\chi: X\to \{-1, 0, 1\}$ that assigns $\pm 1$ to at least $|X|/10$ elements in $X$, and
	\[\max_{A_0\in \mathcal{A}_\vN}|\chi(A_0\cap X)| \leq C_d(N_1\cdots N_d)^{\frac{1}{2d+2}}\cdot \left(\frac{|X|}{N_1\cdots N_d}\right)^{c_d\delta}.\]
	\end{proposition}
	\begin{proof}
	    The general proof strategy is the same as in the proof of Proposition~\ref{prop:first-partial-coloring}. Without loss of generality we may assume $N_1\leq \cdots \leq N_d$. We would like to find a function $b: \N \to (0, \infty)$ such that 
	\begin{equation}\label{eqn:prop-almost-cube-target}
	\sum_{s: s = 2^t \leq N_d} f(s, X) g\left(\frac{b(s)}{\sqrt{s}}\right) \leq |X|/5.\end{equation}
	Here we sum only over $s\leq N_d$ as there is no congruence class of $X\subseteq [N_1]\times \cdots \times [N_d]$ of length greater than $N_d$. 
To estimate $f(s, X)/|X|$, we know that Corollary~\ref{cor:counting-subsets-1} gives an upper bound for $1\leq s \leq N_1$. For $N_1 < s \leq N_d$, we apply Lemma~\ref{lem:counting-subsets-by-U} and Lemma~\ref{lem:counting-U-1} and get that
	\begin{equation}\label{eqn:large-s}
	    \frac{f(s, X)}{|X|} \leq \frac{1}{s}\prod_{i=1}^d\left(4\frac{N_i}{s}+1\right) \leq \frac{1}{s}\cdot 5\frac{N_d}{s}\cdot \prod_{i=1}^{d-1}5\frac{N_i}{N_1} = \frac{5^dN_1\cdots N_d}{N_1^{d-1}}\cdot \frac{1}{s^2}.
	\end{equation}
	Here we use the inequalities $4\frac{N_d}{s}+1\leq 5\frac{N_d}{s}$ as $s \leq N_d$, and $4\frac{N_i}{s}+1\leq 5\frac{N_i}{N_1}$ as $s > N_1$ and $N_i\geq N_1$. Now combining with Lemma~\ref{lem:counting-subsets-3} applied with $\beta = \frac{\delta}{4(d+1)^2}$, we derive the following inequality:
	\begin{equation}\label{eqn:piecewise-up-for-f/m}
	    \frac{f(s, X)}{|X|} \leq \begin{cases}
	    \frac{1}{s^2}\frac{5^dN_1\cdots N_d}{N_1^{d-1}}  & \mbox{if }N_1 < s \leq N_d \\
	    \frac{1}{s^{d+1}}5^dN_1\cdots N_d & \mbox{if } 1\leq s < (N_1\cdots N_d)^\frac{1}{d+1}\rho^\frac{\delta}{4^d(d+1)}\mbox{ or }N_1\rho^\frac{\delta}{4(d+1)^2}< s \leq N_1\\
	    \frac{1}{s^{d+1}}C_dN_1\cdots N_d\rho^{\frac{\delta}{4^{d+1}(d+1)^2\cdot (d+2)!}} & \mbox{if }(N_1\cdots N_d)^\frac{1}{d+1}\rho^\frac{\delta}{4^d(d+1)}\leq s \leq N_1\rho^\frac{\delta}{4(d+1)^2}
	    \end{cases}
	\end{equation}
	for $\rho := \frac{|X|}{N_1\cdots N_d}$. Here we applied \eqref{eqn:large-s} on the first range, Corollary~\ref{cor:counting-subsets-1} on the second, and Lemma~\ref{lem:counting-subsets-3} on the third with $C = C_02^{d^3}5^d$ for some absolute constant $C_0$.
	
	We denote $K_1 = 15\frac{5^dN_1\cdots N_d}{N_1^{d-1}}$, $K_2 = 15\cdot 5^dN_1\cdots N_d$, and $K_3 = 15CN_1\cdots N_d\rho^{\frac{\delta}{4^{d+1}(d+1)^2\cdot (d+2)!}}$ for simplicity. Applying Lemma~\ref{lem:choice-of-b} three times, if we define with $c_1 = 2410$ and $c_2 = 10d+2400$
	\begin{equation}\label{eqn:b1}
	    b_1(s) = \begin{cases}c_1\sqrt{s} \cdot \left(sK_1^{-\frac{1}{2}}\right)^{-1} &\mbox{if } s \geq K_1^{\frac{1}{2}} \\ c_1\sqrt{s} \cdot \left(sK_1^{-\frac{1}{2}}\right)^{-0.1} &\mbox{if } s  <  K_1^{\frac{1}{2}}\end{cases},
	\end{equation}\label{eqn:b23}
	and for $i = 2$ and $3$ 
	\begin{equation}
	    b_i(s) = \begin{cases}c_2\sqrt{s} \cdot \left(sK_i^{-\frac{1}{d+1}}\right)^{-1} &\mbox{if } s \geq K_i^{\frac{1}{d+1}} \\ c_2\sqrt{s} \cdot \left(sK_i^{-\frac{1}{d+1}}\right)^{-0.1} &\mbox{if } s  <  K_i^{\frac{1}{d+1}}\end{cases},
	\end{equation}
	then as $s$ taking values in powers of $2$, we have the following three inequalities:
	\begin{align}
	    \label{eqn:b1-entropy}
	    \sum_{N_1 < s\leq N_d}\frac{K_1}{s^2}g\left(\frac{b_1(s)}{\sqrt{s}}\right) &\leq 1,\\
	    \label{eqn:b2-entropy}
	    \sum_{1\leq s < (N_1\cdots N_d)^\frac{1}{d+1}\rho^\frac{\delta}{4^d(d+1)}}\frac{K_2}{s^{d+1}}g\left(\frac{b_2(s)}{\sqrt{s}}\right) + \sum_{N_1\rho^\frac{\delta}{4}< s \leq N_1}\frac{K_2}{s^{d+1}}g\left(\frac{b_2(s)}{\sqrt{s}}\right)&\leq 1,\\
	    \label{eqn:b3-entropy}
	    \sum_{(N_1\cdots N_d)^\frac{1}{d+1}\rho^\frac{\delta}{4^d(d+1)}\leq s \leq N_1\rho^\frac{\delta}{4}}\frac{K_3}{s^{d+1}}g\left(\frac{b_3(s)}{\sqrt{s}}\right)&\leq 1.
	\end{align}
	If we apply \eqref{eqn:piecewise-up-for-f/m} to \eqref{eqn:b1-entropy}, \eqref{eqn:b2-entropy}, and \eqref{eqn:b3-entropy}, we derive that
	\begin{align}
	    \label{eqn:b1-entropy-1}
	    \sum_{N_1 < s\leq N_d}\frac{f(s, X)}{|X|}g\left(\frac{b_1(s)}{\sqrt{s}}\right) &\leq \frac{1}{15},\\
	    \label{eqn:b2-entropy-1}
	    \sum_{1\leq s < (N_1\cdots N_d)^\frac{1}{d+1}\rho^\frac{\delta}{4^d(d+1)}}\frac{f(s, X)}{|X|}g\left(\frac{b_2(s)}{\sqrt{s}}\right) + \sum_{N_1\rho^\frac{\delta}{4(d+1)^2}< s \leq N_1}\frac{f(s, X)}{|X|}g\left(\frac{b_2(s)}{\sqrt{s}}\right)&\leq 1,\\
	    \label{eqn:b3-entropy-1}
	    \sum_{(N_1\cdots N_d)^\frac{1}{d+1}\rho^\frac{\delta}{4^d(d+1)}\leq s \leq N_1\rho^\frac{\delta}{4(d+1)^2}}\frac{f(s, X)}{|X|}g\left(\frac{b_3(s)}{\sqrt{s}}\right)&\leq 1.
	\end{align}
	Therefore, we may define $b$ in terms of $b_1, b_2, b_3$ as following so that \eqref{eqn:prop-almost-cube-target} is satisfied:
	\begin{equation}
	    b(s) = \begin{cases}
	    b_1(s)  & \mbox{if }N_1 < s \leq N_d \\
	    b_2(s) & \mbox{if } 1\leq s < (N_1\cdots N_d)^\frac{1}{d+1}\rho^\frac{\delta}{4^d(d+1)}\mbox{ or }N_1\rho^\frac{\delta}{4(d+1)^2}< s \leq N_1\\
	    b_3(s) & \mbox{if }(N_1\cdots N_d)^\frac{1}{d+1}\rho^\frac{\delta}{4^d(d+1)}\leq s \leq N_1\rho^\frac{\delta}{4(d+1)^2}
	    \end{cases}.
	\end{equation}
	Hence by Lemma~\ref{lem:partial-coloring-lemma}, we conclude that there exists $\chi:X\to \{-1, 0, 1\}$ that assigns $\pm 1$ to at least $|X|/10$ elements in $X$ such that $|\chi(S)|\leq b(|S|)$ for any $S\in \mathcal{C}_X$. By Lemma~\ref{lem:decomposition}, we know that $|\chi(A)| \leq 2\sum_s b(s)$ for any $A\in \mathcal{A}_X$. We shall estimate $\sum_s b(s)$ on five intervals using Lemma~\ref{lem:sum-b}:
	\begin{align}
	    \label{eqn:int-1}
	    \sum_{N_1 < s \leq N_d}b_1(s) &\leq 5c_1K_1^{\frac{1}{2}}N_1^{-\frac{1}{2}},\\
	    \label{eqn:int-2}
	    \sum_{1\leq s < (N_1\cdots N_d)^\frac{1}{d+1}\rho^\frac{\delta}{4^d(d+1)}}b_2(s) & \leq 5c_2K_2^\frac{1}{10d+10}(N_1\cdots N_d)^\frac{2}{5d+5}\rho^\frac{2\delta}{5\cdot 4^d(d+1)},\\
	    \label{eqn:int-3}
	    \sum_{N_1\rho^\frac{\delta}{4(d+1)^2}< s \leq N_1}b_2(s) &\leq 5c_2K_2^\frac{1}{d+1}N_1^{-\frac{1}{2}}\rho^{-\frac{\delta}{8(d+1)^2}},\\
	    \label{eqn:int-4}
	    \sum_{(N_1\cdots N_d)^\frac{1}{d+1}\rho^\frac{\delta}{4^d(d+1)}\leq s < K_3^{\frac{1}{d+1}}}b_3(s) & \leq 5c_2K_3^\frac{1}{2d+2},\\
	    \label{eqn:int-5}
	    \sum_{K_3^{\frac{1}{d+1}} \leq s \leq N_1\rho^\frac{\delta}{4(d+1)^2}}b_3(s) &\leq 5c_2K_3^\frac{1}{2d+2}.
	\end{align}
	Note that $N_1 \geq (N_1\cdots N_d)^\frac{1}{d+1-\delta}$. As $X$ is nonempty, we have $\rho \geq \frac{1}{N_1\cdots N_d}$. For \eqref{eqn:int-1} we have
	\begin{equation}\label{eqn:int-1-sim}
	    5c_1K_1^{\frac{1}{2}}N_1^{-\frac{1}{2}} \leq 5^{\frac{d}{2}+2}c_1(N_1\cdots N_d)^{\frac{1}{2}\left(1-\frac{d}{d+1-\delta}\right)}\leq 5^{\frac{d}{2}+2}c_1(N_1\cdots N_d)^{\frac{1}{2d+2}}\rho^{\frac{d\delta}{2(d+1)(d+1-\delta)}}.
	\end{equation}
	Here in the first inequality we use $\sqrt{15\cdot 5^d}\leq 5^{\frac{d}{2}+1}$ and $N_1 \geq (N_1\cdots N_d)^\frac{1}{d+1-\delta}$, and in the second we use $\rho \geq \frac{1}{N_1\cdots N_d}$. Similarly we can estimate the right hand sides of \eqref{eqn:int-2}, \eqref{eqn:int-3}, \eqref{eqn:int-4}, \eqref{eqn:int-5} as follows:
	\begin{equation}\label{eqn:int-2-sim}
	    5c_2K_2^\frac{1}{10d+10} = 5c_2\cdot (5^d15)^\frac{1}{10d+10}(N_1\cdots N_d)^\frac{1}{2d+2}\rho^\frac{2\delta}{5\cdot 4^d(d+1)} \leq 10c_2(N_1\cdots N_d)^\frac{1}{2d+2}\rho^\frac{2\delta}{5\cdot 4^d(d+1)},
	\end{equation}
	\begin{equation}\label{eqn:int-3-sim}
	\begin{split}
	    5c_2K_2^\frac{1}{d+1}N_1^{-\frac{1}{2}}\rho^\frac{-\delta}{8(d+1)^2} &\leq 50c_2(N_1\cdots N_d)^{\frac{1}{2d+2}} (N_1\cdots N_d)^{\frac{1}{2d+2}- \frac{1}{2d+2-2\delta}}\rho^\frac{-\delta}{8(d+1)^2}\\
	    & \leq 
	    50c_2(N_1\cdots N_d)^{\frac{1}{2d+2}}\rho^{\frac{\delta}{2(d+1)(d+1-\delta)}-\frac{\delta}{8(d+1)^2}} \\
	    & \leq 50c_2(N_1\cdots N_d)^\frac{1}{2d+2}\rho^\frac{\delta}{4(d+1)^2},
	    \end{split}
	\end{equation}
	\begin{equation}\label{eqn:int-45-sim}
	    5c_2K_3^\frac{1}{2d+2}\leq 10c_2C^\frac{1}{2d+2}(N_1\cdots N_d)^\frac{1}{2d+2}\rho^\frac{\delta}{2\cdot 4^{d+1}(d+1)^3\cdot (d+2)!}.
	\end{equation}
	Now we define $c_d = \frac{1}{2\cdot 4^{d+1}(d+1)^3\cdot (d+2)!}$. The exponents of $\rho$ in \eqref{eqn:int-1-sim}, \eqref{eqn:int-2-sim}, \eqref{eqn:int-3-sim} and \eqref{eqn:int-45-sim} are at least $c_d\delta$. Putting \eqref{eqn:int-1-sim} into \eqref{eqn:int-1}, \eqref{eqn:int-2-sim} into \eqref{eqn:int-2}, \eqref{eqn:int-3-sim} into \eqref{eqn:int-3}, and \eqref{eqn:int-45-sim} into \eqref{eqn:int-4} and \eqref{eqn:int-5} and summing them together, we derive that
	\[\sum_{s}b(s) \leq \left(5^{\frac{d}{2}+2}c_1 + 10c_2 + 50c_2 + 10c_2C^\frac{1}{2d+2} + 10c_2C^\frac{1}{2d+2}\right)\cdot (N_1\cdots N_d)^\frac{1}{2d+2}\rho^{c_d\delta}.\]
	Hence if we set $C_d = 2\left(5^{\frac{d}{2}+2}c_1 + 60c_2+20c_2C^\frac{1}{2d+2}\right)$, then for any $X$ there exists $\chi:X\to \{-1, 0, 1\}$ that assigns $\pm 1$ to at least $|X|/10$ elements in $X$ such that 
	\[\max_{A_0\in \mathcal{A}_\vN}|\chi(A_0\cap X)| = \max_{A\in \mathcal{A}_X}|\chi(A)|\leq C_d(N_1\cdots N_d)^\frac{1}{2d+2}\left(\frac{|X|}{N_1\cdots N_d}\right)^{c_d\delta},\]
	as desired.
	\end{proof}
	\begin{remark}
	The argument gives $C_d = 2^{O(d^2)}$ and $c_d = 2^{-O(d\log d)}$.
	\end{remark}
	\begin{corollary}\label{cor: almost cube partial color}
	Let $d$ be a positive integer. There exist constants $C_d$ and $c_d$ such that the following holds. Let $N_1, N_2, \dots, N_d$ be positive integers satisfying $N_1\cdots N_d \leq (\min_{1\leq i\leq d} N_i)^{d+1-\delta}$ for some $\delta\in (0, 1]$. For any nonempty $X\subseteq [N_1]\times \cdots \times [N_d]$, there exists a coloring $\chi: X\to \{-1, 1\}$ with
	\[\max_{A_0\in \mathcal{A}_\vN}|\chi(A_0\cap X)| \leq C_d\frac{1}{\delta}(N_1\cdots N_d)^{\frac{1}{2d+2}}\cdot \left(\frac{|X|}{N_1\cdots N_d}\right)^{c_d\delta}.\]
	\end{corollary}
	\begin{proof}
	    Let $C_d', c_d'$ be the constants in Proposition~\ref{prop:partial-coloring-almost-cube}. Then we set $c_d = \min(c_d', 1)$. Starting from $X_1 = X$, for each $i\geq 0$, we partially color $X_i$ by using Proposition~\ref{prop:partial-coloring-almost-cube} and let $X_{i+1}$ be the set of uncolored points. Suppose that we do this for $K$ iterations such that $X_{K+1} = \emptyset$. Since $|X_i| \leq (0.9)^{i-1}|X|$ it follows that the total discrepancy of this coloring is at most
	    \[\sum_{i=1}^K C_d'(N_1\cdots N_d)^\frac{1}{2d+2}\left(\frac{|X_i|}{N_1\cdots N_d}\right)^{c_d\delta} \leq (N_1\cdots N_d)^\frac{1}{2d+2}\left(\frac{|X|}{N_1\cdots N_d}\right)^{c_d\delta}\cdot C_d'\sum_{i=1}^K \left((0.9)^{c_d\delta}\right)^{i-1}.\]
	    Finally noting that $\delta \leq 1$, we have $1-(0.9)^{c_d\delta} \geq \frac{c_d\delta}{10}$ as $c_d\delta \in (0, 1]$. Therefore we have
	    \[\sum_{i=1}^K \left((0.9)^{c_d\delta}\right)^{i-1} \leq \frac{1}{1-(0.9)^{c_d\delta}} \leq \frac{10}{c_d\delta}.\]
	    We may set $C_d = \frac{10C_d'}{c_d}$ to get the desired inequality.
	\end{proof}
\begin{proof}[Proof of the upper bound of Theorem~\ref{thm:almost-cubes}]
 	By using Corollary~\ref{cor: almost cube partial color}
 	and taking $X = [N_1]\times \cdots \times [N_d]$, we have $\disc(\mathcal{A}_\vN) \leq C_d\frac{(N_1\cdots N_d)^{\frac{1}{2d+2}}}{\delta}$ for some $C_d = 2^{O(d^2)}$ which proves Theorem~\ref{thm:almost-cubes}.
\end{proof}

\section{A better estimate on the number of sets in the decomposition}\label{sec:geometry}
    In this section, we prove Lemma~\ref{lem:counting-subsets-3}, which improves upon Lemma~\ref{lem:counting-U-1}. 
    Recall that $U^d(X, \vb, s)$ is the set of elements in $X$ for which there are at least $s$ elements of $X$ in the same residue class mod $\vb$. 

To prove Lemma~\ref{lem:counting-subsets-3}, we induct on $d$.
	We need the following two results regarding lattice points.
	\begin{lemma}[Minkowski's theorem, see e.g.~Section III.2.2 in \cite{Cassels}]\label{thm:Minkowski}
	    Let $X\subseteq \R^d$ be a point set of volume $V(X)$ which is symmetric about the origin and convex. Let $\Gamma\subseteq \R^d$ be a $d$-dimensional lattice of determinant $\det(\Gamma)$. If $V(X) > 2^d\det(\Gamma)$, then $X\cap \Gamma$ contains a pair of distinct points $\pm \vx$.
	\end{lemma}
	
	\begin{lemma}[Lenstra-Lenstra-Lov\'asz Basis Reduction \cite{LLL}]\label{thm:LLL}
	    Let $\Gamma\subseteq \R^d$ be a $d$-dimensional lattice of determinant $\det(\Gamma)$. There exists a basis $\vx_1, \dots, \vx_d$ of $\Gamma$ such that
	    \[\det(\Gamma) \leq \prod_{i=1}^n \|\vx_i\|_2 \leq 2^{\frac{d(d-1)}{4}}\det(\Gamma).\]
	\end{lemma}

	\begin{remark}
	The inequality above is true for any reduced basis (see \cite[Proposition 1.6]{LLL}); the existence of which is guaranteed by an algorithm which transforms any given basis to a reduced one (see  \cite[Proposition 1.26]{LLL}). It will not be important for us what the definition of a reduced basis is. 
	\end{remark}
	
	\begin{corollary}\label{cor:LLL}
	Let $\Gamma\subseteq \R^d$ be a $d$-dimensional lattice. Suppose that $V_0 > 0$ is a real number such that the following holds: for every set $P\subseteq \R^d$ of volume $V_d(P) > V_0$ that is symmetric about the origin and convex, $P\cap \Gamma$ contains a nonzero point. Then there exists a basis $\vx_1, \dots, \vx_d$ of $\Gamma$ such that
	\[\prod_{i=1}^n \|\vx_i\|_2 \leq 2^{\frac{d(d-1)}{4}-d}V_0.\]
	\end{corollary}
	\begin{proof}
	Let $\vy_1, \dots, \vy_d$ be a basis of $\Gamma$. Note that the fundamental parallelepiped of $\Gamma$ defined as $\{\mathbf{a}\cdot \vy: \mathbf{a}\in (0, 1)^{d-1}\}$ where $\vy := (\vy_1, \dots, \vy_d)$ contains no nonzero vector of $\Gamma$. By translation, there is no nonzero vector of $\Gamma$ in $P:= \{\mathbf{a}\cdot \vy : \mathbf{a}\in (-1, 1)^{d-1}\}$, and from the condition we know that $V_d(P) \leq V_0$. Note that $V_d(P) = 2^d\det(\Gamma)$, so $\det(\Gamma) \leq 2^{-d}V_0$. By Lemma~\ref{thm:LLL}, there exists a basis $\vx_1, \dots, \vx_d$ such that
	\[\prod_{i=1}^n\|\vx_i\|_2 \leq 2^{\frac{d(d-1)}{4}}\det(\Gamma) \leq 2^{\frac{d(d-1)}{4}-d}V_0,\]
	which completes the proof.
	\end{proof}
	
	Before we start to prove Lemma~\ref{lem:counting-subsets-3}, we introduce some standard notation. For a map $\phi: X\to Y$ and subsets $A\subseteq X$ and $B\subseteq Y$, let $\phi(A)$ be the set of images $\{\phi(a): a\in A\}\subseteq Y$, and let $\phi^{-1}(B)$ be the set of preimages $\{x\in X: \phi(x)\in B\}\subseteq X$.

	We next prove a geometric lemma which shows that, given a vector $\vb \in \mathbb{Z}^d$ whose coordinates have greatest common divisor $1$, there is a linear map from $\Z^d$ to $\Z^{d-1}$ which has full rank with null space generated by $\vb$, and maps a grid into another grid with similar size. 
	
	\begin{lemma}\label{lem:projection-coprime}
	Let $d \geq 2$ be a positive integer, and $N_1, N_2, \dots, N_d$ be positive integers. Suppose that $\vb = (b_1, \dots, b_d)\in \Z^d$ is a nonzero vector satisfying that $\gcd(b_1, \dots, b_d) = 1$ and $|b_i| \leq N_i$ for all $1\leq i\leq d$. Then there exists a linear map $f_\vb: \Z^{d} \to \Z^{d-1}$ so that the following two conditions holds.
	\begin{enumerate}
		\item \label{item:nullspace-projection-coprime}For any $\vx_1, \vx_2 \in \Z^d$, $f_\vb(\vx_1) = f_\vb(\vx_2)$ if and only if $\vx_1 - \vx_2 = k\vb$ for some $k\in \Z$.
		\item \label{item:range-projection-coprime}There exist positive integers $N_1^*, N_2^*, \dots, N_{d-1}^* \geq \min_{1\leq i\leq d} N_i$ so that
		\begin{equation}\label{eqn:range-projection-coprime} \frac{1}{2} N_1\cdots N_d \cdot \left(\max_{1\leq i\leq d}\frac{|b_i|}{N_i}\right)\leq N_1^*\cdots N_{d-1}^* \leq 2^{d^2} N_1\cdots N_d \cdot \left(\max_{1\leq i\leq d}\frac{|b_i|}{N_i}\right)\end{equation}
		and $f_\vb([N_1]\times \cdots \times [N_d])\subseteq [N_1^*] \times \cdots \times [N_{d-1}^*]$.
		\end{enumerate}
	\end{lemma}
	\begin{proof}
	We may write $f_\vb(\vx) = M\vx + \mathbf{v}$ for some $M\in \Z^{(d-1)\times d}$ and $\mathbf{v}\in \Z^{d-1}$ to be chosen later. Condition \eqref{item:nullspace-projection-coprime} says that $M$ is full rank, with null space generated by $\vb$.
	
	We regard the rows of $M$ as vectors $\vr_1, \dots, \vr_{d-1}\in \Z^d$. For each $1\leq j\leq d-1$, if we write $\vr_j = (r_{j1}, \dots, r_{jd})$, then we define $\vr^*_j := (r_{j1}N_1, \dots, r_{jd}N_d) \in \Lambda$ where $\Lambda := N_1\Z\times \cdots \times N_d\Z$. Condition \eqref{item:nullspace-projection-coprime} is equivalent to saying that the vectors $\vr_j$ in $\Z^d$ for $1\leq j\leq d-1$ are linearly independent and $\vr_j \cdot \vb = 0$ for all $1\leq j\leq d-1$. In terms of $\vr^*_j$, this is equivalent to $\vr^*_j\cdot \vb^* = 0$ for $\vb^* := (\frac{b_1}{N_1}, \dots, \frac{b_d}{N_d})$, and $\vr^*_1, \dots, \vr^*_{d-1}$ are linearly independent vectors. The following claim allows us to find these vectors whose product of $\ell_2$-norms is small.
	\begin{claim}\label{claim:find-short-basis}
	There exists linearly independent vectors $\vr^*_1, \dots, \vr^*_{d-1}\in \Lambda$ that satisfy $\vr^*_j\cdot \vb^* = 0$ for each $1\leq j\leq d-1$, and 
	\begin{equation}\label{eqn:LLL}\prod_{j=1}^{d-1}\|\vr^*_j\|_2 \leq 2^{\frac{(d-1)(d-2)}{4}} \cdot  \|\vb^*\|_2N_1N_2\cdots N_d,\end{equation}
	\end{claim}
	\begin{proof}[Proof of Claim~\ref{claim:find-short-basis}]
	Consider the subspace of $\R^d$ defined by $\langle \vb^*\rangle^\perp := \{\vx\in \R^d: \vx \cdot \vb^* = 0\}$ which has dimension $d-1$. The intersection $\Lambda^* := \Lambda \cap \langle \vb^*\rangle^\perp$ is a lattice in $\langle \vb^*\rangle^\perp$. As $\vb\in \Z^d$ is a nonzero vector with integer entries, there exist linearly independent vectors $\vr_1, \dots, \vr_{d-1}\in \Z^d$ that satisfy $\vr_j\cdot \vb = 0$ for each $1\leq j\leq d-1$. Hence we can find $d-1$ linearly independent vectors $\vr_j^*\in \Lambda\cap \langle \vb^*\rangle^\perp$ defined by $\vr_j^* = (r_{j1}N_1,\dots, r_{jd}N_d)$ where $\vr_j = (r_{j1}, \dots, r_{jd})$ for $1\leq j\leq d-1$. The linear independence of $\{\vr_j^*\}_{j=1}^{d-1}$ follows from the linear independence of $\{\vr_j\}_{j=1}^{d-1}$, while $\vr_j^*\cdot \vb^* = \vr_j \cdot \vb = 0$ for each $1\leq j\leq d-1$. Thus we conclude that $\Lambda^*$ is a $(d-1)$-dimensional lattice.
	
	We next consider some geometric properties of $\Lambda^*$ as a subset of the $(d-1)$-dimensional Euclidean space $\langle \vb^*\rangle^\perp$. Let $P$ be a subset of $\langle \vb^*\rangle^\perp$ that is symmetric about the origin and convex. Now we consider the set $X := \{\mathbf{p} + a\vb^*: \mathbf{p}\in P, a\in (-1/\|\vb\|_2^2, 1/\|\vb\|_2^2)\}$. We could equivalently phrase this as: $X$ contains all the point $\vx$ that satisfies $\vx\cdot \vb^* \in (-1, 1)$, and the projection of $\vx$ onto the hyperspace $\langle\vb^*\rangle$ is in $P$. Therefore from this geometric interpretation we know that $V_d(X) = \frac{2}{\|\vb\|_2}V_{d-1}(P)$ where $V_d$ denotes the $d$-dimensional volume. Meanwhile, we see that as $P$ and $\{a\vb^*: a\in (-1/\|\vb\|_2^2, 1/\|\vb\|_2^2)\}$ are both convex and symmetric about the origin, their Minkowski sum $X$ is also convex and symmetric about the origin. By Minkowski's theorem, Lemma \ref{thm:Minkowski}, if $V_d(X) > 2^d\det(\Lambda)$, then there is a nonzero point $\vx\in \Lambda\cap X$. Note that any point $\vx\in \Lambda$ satisfies that $\vx\cdot \vb^*$ is an integer, while any point $\vb\in X$ satisfies that $\vb\cdot \vb^* \in (-1, 1)$, we know that if $\vx \in \Lambda\cap X$, then $\vx\in \langle \vb^*\rangle^\perp$. Note that $\Lambda^* = \Lambda\cap \langle \vb^*\rangle^\perp$ and $P = X\cap \langle \vb^*\rangle^\perp$, this means $\vx \in \Lambda^*\cap P$. In summary, if $V_{d-1}(P) > 2^{d-1}\|\vb\|_2\det(\Lambda) = 2^{d-1}\|\vb\|_2N_1\cdots N_d$, then $P\cap \Lambda^*$ contains a nonzero point.

    Therefore we may apply Corollary~\ref{cor:LLL} with dimension $d-1$, lattice $\Gamma$, and $V_0=2^{d-1}\|\vb\|_2N_1\cdots N_d$, and we obtain that there exists a basis $\vr_1^*,\dots, \vr_{d-1}^*$ such that
	\begin{equation*}\prod_{j=1}^{d-1}\|\vr^*_j\|_2 \leq 2^{\frac{(d-1)(d-2)}{4}} \cdot  \|\vb^*\|_2N_1N_2\cdots N_d,\end{equation*}
	so we have these linearly independent vectors as expected.
	\end{proof}
	
	From the set of vectors $\{\vr_j^*\}_{j=1}^{d-1}$ whose existence is guaranteed by Claim~\ref{claim:find-short-basis}, we obtain the set of vectors $\{\vr_j\}_{j=1}^{d-1}$, which are the row vectors of $M$. Then condition \eqref{item:nullspace-projection-coprime} is satisfied, since they are $d-1$ linearly independent vectors in $\Z^d$ and satisfy that $\vr_j^*\cdot \vb^* = 0$ for each $1\leq j\leq d-1$.
	
	For each $1\leq j\leq d-1$, we know that for any $\vx = (x_1, \dots, x_d)\in \Z^d$, $(M\vx)_j = \vr_j\cdot \vx$. Note that $\vr_j = (r_{j1}, \dots, r_{jd})$. We have that whenever $\vx\in [N_1]\times \cdots \times [N_d]$,
	\begin{equation}\label{eqn:imageisingrid}|\vr_j\cdot \vx| = \left|\sum_{i = 1}^d r_{ji}x_i\right| \leq \sum_{i=1}^d |r_{ji}N_i| = \|\vr^*_j\|_1.\end{equation}
	Let $N_j^* = 3\|\vr^*_j\|_1$ and $\mathbf{v} = (2\|\vr^*_j\|_1)_{j=1}^{d-1}$. Observe that this choice of parameters together with (\ref{eqn:imageisingrid}) ensures that $f_\vb([N_1]\times \cdots \times [N_d])\subseteq [N_1^*]\times \cdots \times [N_{d-1}^*]$. For each $1\leq j\leq d-1$, as $\vr_j\in \Z^d$ is nonzero, we have $r_{ji}$ is nonzero for some $i$, and hence $N^*_j \geq \|\vr^*_j\|_1 \geq |r_{ji}N_i| \geq N_i \geq \min_{1\leq i\leq d}N_i$. 
	
	Also, by condition \eqref{item:nullspace-projection-coprime}, elements in $f_\vb^{-1}(\vx^*)$ differ by multiples of $\vb$. Note that there are at most $\frac{1}{\|\vb^*\|} + 1 \leq \frac{2}{\|\vb^*\|}$ such elements in $[N_1]\times \cdots \times [N_d]$ for each fixed $\vx^*\in [N_1^*]\times \cdots \times [N_{d-1}^*]$. We have
	\[\frac{2}{\|\vb^*\|}N_1^* \cdots N_{d-1}^* \geq N_1\cdots N_d.\]
	It remains to show the other half of the inequality in \eqref{eqn:range-projection-coprime}. With $N_j^*$ as above, we have
	\[N_1^*\cdots N_{d-1}^* = \prod_{j=1}^{d-1}3 \|\vr^*_j\|_1 \leq \prod_{j=1}^{d-1}3\sqrt{d} \|\vr^*_j\|_2 \leq 3^{d-1}(\sqrt{d})^{d-1} 2^{\frac{(d-1)(d-2)}{4}}\|\vb^*\|_2N_1N_2\cdots N_d.\]
	Using that $\|\vb^*\|_2\leq \sqrt{d} \|\vb^*\|_{\infty}$, we have
	\[3^{d-1}(\sqrt{d})^{d-1} 2^{\frac{(d-1)(d-2)}{4}}\|\vb^*\|_2N_1N_2\cdots N_d \leq 2^{d^2}N_1N_2\cdots N_d\|\vb^*\|_{\infty},\]
	so condition \eqref{item:range-projection-coprime} is also satisfied. Here we use that $3^{d-1}(\sqrt{d})^{d} 2^{\frac{(d-1)(d-2)}{4}} \leq 2^{d^2}$ for $d\geq 2$.\end{proof}
	
	For any $\vb = (b_1, \dots, b_d)\in \Z^d\setminus \{\vzero\}$ whose entries are not coprime, we may apply the lemma above to $\vb/\gcd(b_1, \dots, b_d)$ instead. This gives the following corollary.
	\begin{corollary}\label{cor:projection}
	Let $d \geq 2$ be an integer, and $N_1, N_2, \dots, N_d$ be positive integers. Suppose that $\vb = (b_1, \dots, b_d)\in \Z^d$ is a nonzero vector satisfying that $\lambda = \lambda(\vb):= \max_{1\leq i\leq d} \frac{|b_i|}{\gcd(b_1, \dots, b_d)N_i} \leq 1$. Then there exists a linear map $f_\vb \in \Z^{(d-1)\times d}$ so that the following two conditions holds.
	\begin{enumerate}
		\item \label{item:nullspace-projection}For any $\vx_1, \vx_2 \in \Z^d$, $f_\vb(\vx_1) = f_\vb(\vx_2)$ if and only if $\vx_1 - \vx_2 = k\vb$ for some $k\in \Q$.
		\item \label{item:range-projection}There exist positive integers $N_1^*, N_2^*, \dots, N_{d-1}^* \geq \min_{1\leq i\leq d} N_i$ so that $\frac{1}{2}\leq \frac{N_1^*\cdots N_{d-1}^*}{\lambda N_1\cdots N_d} \leq 2^{d^2}$ and $f_\vb([N_1]\times \cdots \times [N_d])\subseteq [N_1^*] \times \cdots \times [N_{d-1}^*]$.
		\end{enumerate}
	\end{corollary}
	
	The linear map in the corollary above is the main tool for reduction from $\Z^d$ to $\Z^{d-1}$. We have the following simple relation between $f_\vb$ and $U^d(X, \vb, s)$.
	\begin{lemma}\label{lem:preimage-of-projection}
	Let $d\geq 2$ be an integer, and $N_1, \dots, N_d$ be positive integers. Let $X\subseteq [N_1]\times \cdots \times [N_d]$, let $\vb = (b_1, \dots, b_d)\in \Z^d\setminus \{\vzero\}$ with $\lambda = \lambda(\vb):= \max_{1\leq i\leq d} \frac{|b_i|}{\gcd(b_1, \dots, b_d)N_i}  \leq 1$. Let $s \leq \min_{i}N_i$ be a positive integer. Suppose that $f_\vb: \Z^d\to \Z^{d-1}$ is a linear map satisfying \eqref{item:nullspace-projection} in Corollary~\ref{cor:projection}. Then each element in $f_\vb(U^d(X, \vb, s))$ has at least $s$ and at most $\frac{2}{\lambda}$ preimages under $f_\vb$ in $U^d(X, \vb, s)$.
	\end{lemma}
	\begin{proof}
	as $f_\vb$ satisfies \eqref{item:nullspace-projection}, we know that if $f_\vb(\vx_1) = f_\vb(\vx_2)$, then $\vx_1 - \vx_2$ is a multiple of $\vb/\gcd(b_1, \dots, b_d)$. From the definition of $\lambda$, we see that each element in $\Z^{d-1}$ has at most $\frac{1}{\lambda}+1 \leq \frac{2}{\lambda}$ preimages in $U^d(X, \vb, s)\subseteq [N_1]\times \cdots \times [N_d]$.
	
	Note that each element in $U^d(X, \vb, s)$ is in a residue class mod $\vb$ of size at least $s$. Again by condition \eqref{item:nullspace-projection}, elements in the residue class mod $\vb$ get mapped to the same element by $f_\vb$. Therefore, every element in $f_\vb(U^d(X, \vb, s))$ has at least $s$ preimages.
	\end{proof}

	Using Lemma~\ref{lem:preimage-of-projection} we derive the following lemma, which bounds the size of the intersection of two $U^d$ sets.
	\begin{lemma}\label{lem:single-intersection}
	Let $d\geq 2$ be a positive integer, and $N_1, \dots, N_d$ be positive integers. Let $X\subseteq [N_1]\times \cdots \times [N_d]$, let $\vb = (b_1, \dots, b_d)\in \Z^d\setminus \{\vzero\}$ with $\lambda = \lambda(\vb):= \max_{1\leq i\leq d} \frac{|b_i|}{\gcd(b_1, \dots, b_d)N_i}  \leq 1$. Let $s^* \leq s \leq \min_{i}N_i$ be positive integers. Suppose that $f_\vb: \Z^d\to \Z^{d-1}$ satisfies condition \eqref{item:nullspace-projection} in Corollary~\ref{cor:projection}, and $\vb'\in \Z^d$ satisfies that $f_\vb(\vb') \ne f_\vb(\vzero)$. Then
	\begin{equation}\label{eqn:single-intersection}
	|U^d(X, \vb, s) \cap U^d(X, \vb', s)| \leq \frac{s^*-1}{s}|U^d(X, \vb', s)| + \frac{2}{\lambda}|U^{d-1}(f_\vb(U^d(X, \vb, s)), f_\vb(\vb') - f_\vb(\vzero), s^*)|.
	\end{equation}
	\end{lemma}
	\begin{proof}
	For simplicity let $X_1 = U^d(X, \vb, s)$. Note that $f_\vb(X_1)$ is a subset of $\Z^{d-1}$.
	
	Partition the set $U^d(X, \vb', s)$ into nonempty residue classes mod $\vb'$. By definition, we know that each such residue class contains at least $s$ elements, so there are at most $\frac{1}{s}|U^d(X, \vb', s)|$ such residue classes. For each such residue class, there are two cases:
	\begin{itemize}
	\item the residue class contains at most $s^* - 1$ elements in $X_1\cap U^d(X, \vb', s)$;
	\item the residue class contains at least $s^*$ elements in $X_1\cap U^d(X, \vb', s)$.
	\end{itemize}
	We next upper bound the size of $X_1 \cap U^d(X, \vb', s)$. The number of elements in $X_1\cap U^d(X, \vb', s)$ contained in a residue class of the first case is at most $\frac{s^*-1}{s}|U^d(X, \vb', s)|$. It remains to estimate the number of elements contained in a residue class of the second case.
	
	We first show that, if $\vx$ is an element of $X_1\cap U^d(X, \vb', s)$, whose residue class mod $\vb'$ contains at least $s^*$ elements in $X_1\cap U^d(X, \vb', s)$, then
	\begin{equation}\label{eqn: f_b(x)}
	    f_\vb(\vx)\in U^{d-1}(f_\vb(X_1), f_\vb(\vb') - f_\vb(\vzero), s^*).
	\end{equation}
Let $I = \{\vx_1, \dots, \vx_k\}$ be the residue class mod $\vb'$ of $X_1\cap U^d(X, \vb', s)$ with $\vx = \vx_1$. Suppose that $I$ contains at least $s^*$ elements. Note that if $\vx_i = t\vb' + \vx_j$ for some integer $t\ne 0$, then
    $f_\vb(\vx_i) = f_\vb(\vx_j) + t(f_\vb(\vb') - f_\vb(\vzero))$. As $f_\vb(\vb') \ne f_\vb(\vzero)$, we know that $f_\vb(I) = \{f_\vb(\vx_i): 1\leq i\leq k\}$ consists of $k\geq s^*$ distinct elements, whose pairwise differences are multiples of $(f_\vb(\vb') - f_\vb(\vzero))$. In particular, there are at least $k \geq s^*$ elements in $f_\vb(X_1)$ that are congruent to $f_\vb(\vx)$ mod $(f_\vb(\vb') - f_\vb(\vzero))$. This proves \eqref{eqn: f_b(x)}.
	
	Note that by Lemma~\ref{lem:preimage-of-projection}, each element in $U^{d-1}(f_\vb(X_1), f_\vb(\vb') - f_\vb(\vzero), s^*)$ has at most $\frac{2}{\lambda}$ preimages in $X_1\cap U^d(X, \vb', s)$. Therefore, the number of elements contained in a residue class of the second case is at most $\frac{2}{\lambda}|U^{d-1}(f_\vb(X_1), f_\vb(\vb') - f_\vb(\vzero), s^*)|$. Putting these together, we have \eqref{eqn:single-intersection}.
	\end{proof}

	\begin{lemma}\label{lem:sum-over-intersection-n-dim}
	Let $d\geq 2$ be a positive integer and $N_1, \dots, N_d, s$ be positive integers with $\min_{1\leq i\leq d} N_i > s \geq 2$. Let $X\subseteq [N_1]\times \cdots \times [N_d]$, $\lambda_0\in (0, \frac{1}{s-1})$ and $B\subseteq [-\frac{N_1}{s-1}, \frac{N_1}{s-1}] \times \cdots \times [-\frac{N_d}{s-1}, \frac{N_d}{s-1}]$ such that each $(b_1, \dots, b_d)\in B$ is a nonzero integer point satisfying that $\lambda_0 \leq \max_{1\leq i\leq d} \frac{|b_i|}{\gcd(b_1, \dots, b_d)N_i}$. Let $\vb$ be an integer point in $B$, and $f_\vb:\Z^d\to \Z^{d-1}$ be a map satisfying condition \eqref{item:nullspace-projection} in Corollary~\ref{cor:projection}, and $s^*\leq s$ be a positive integer. Then we have
	\begin{equation}\begin{split}\label{eqn:sum-over-intersection-n-dim}
		\sum_{\vb'\in B}|U^d(X, \vb, s)\cap U^d(X, \vb', s)| & \leq  \frac{4}{s\lambda_0}|U^d(X, \vb, s)| + \frac{s^*-1}{s}\sum_{\vb'\in B}|U^d(X, \vb', s)| \\
		& \quad  + \frac{12}{s\lambda_0^2}\sum_{\vb^*\in \Z^{d-1}\setminus \{\vzero\}}|U^{d-1}(f_\vb(U^d(X, \vb, s)), \vb^*, s^*)|.
	\end{split}
	\end{equation}
	\end{lemma}
	\begin{proof}
	Let us denote $\vb = (b_1, \dots, b_d)$. We first give an upper bound on the number of elements $\vb'\in B$ with $f_\vb(\vb') = f_\vb(\vzero)$. By condition \eqref{item:nullspace-projection}, we know that all such $\vb'$ are given by $k \vb/\gcd(b_1, \dots, b_d)$ for $k \in \Z$. If $k \vb/\gcd(b_1, \dots, b_d)$ is contained in $B\subseteq [-\frac{N_1}{s-1}, \frac{N_1}{s-1}] \times \cdots \times [-\frac{N_d}{s-1}, \frac{N_d}{s-1}]$, then we know that $\frac{|k|\cdot |b_i|}{\gcd(b_1, \dots, b_d)} \leq \frac{N_i}{s-1}$ for all $i$. Hence each choice of $k$ with $k \vb/\gcd(b_1, \dots, b_d)\in B$ satisfies that
	\[|k| \leq \min_{1\leq i\leq d} \frac{N_i\gcd(b_1, \dots, b_d)}{|b_i|(s-1)} \leq \frac{1}{(s-1)\lambda_0} \le  \frac{2}{s\lambda_0}.\]
	Therefore, the number of such $\vb'\in B$ is at most $\frac{4}{s\lambda_0}$ (noting that $k\ne 0$). Thus we have
	\begin{equation}\label{eqn:sum-over-intersection-n-dim-zero-part}
	\sum_{\vb'\in B: f_\vb(\vb') = f_\vb(\vzero)}|U^d(X, \vb, s)\cap U^d(X, \vb', s)| \leq \sum_{\vb'\in B: f_\vb(\vb') = f_\vb(\vzero)}|U^d(X, \vb, s)| \leq \frac{4}{s\lambda_0}|U^d(X, \vb, s)|.
	\end{equation}

	Now we consider the summation over all $\vb'\in B$ with $f_\vb(\vb') \ne f_\vb(\vzero)$. For each such $\vb'$, by Lemma~\ref{lem:single-intersection},
	\begin{equation}\label{eqn:two-U-intersection-n-dim}
		|U^d(X, \vb, s)\cap U^d(X, \vb', s)| \leq \frac{s^*-1}{s}|U^d(X, \vb', s)| + \frac{2}{\lambda}|U^{d-1}(f_\vb(U^d(X, \vb, s)), f_\vb(\vb') - f_\vb(\vzero), s^*)|.
	\end{equation}
	Observe that
	\[\sum_{\vb'\in B: f_\vb(\vb') \ne f_\vb(\vzero)}|U^d(X, \vb', s)| \leq \sum_{\vb'\in B}|U^d(X, \vb', s)|.\]
	Now, note that $\vb'\in B \subseteq [-\frac{N_1}{s-1}, \frac{N_1}{s-1}] \times \cdots \times [-\frac{N_d}{s-1}, \frac{N_d}{s-1}]$. Thus, by a similar argument as above, for each $\vb^*\in \Z^{d-1}\setminus \{\vzero\}$, the number of $\vb'\in B$ with $f_\vb(\vb') - f_\vb(\vzero) = \vb^*$ is at most $\frac{2}{(s-1)\lambda_0} + 1 \leq \frac{6}{s\lambda_0}$. Therefore we have
	\[\sum_{\vb'\in B: f_\vb(\vb')\ne f_\vb(\vzero)}|U^{d-1}(f_\vb(U^d(X, \vb, s)), f_\vb(\vb') - f_\vb(\vzero), s^*)| \leq \frac{6}{s\lambda_0} \cdot \sum_{\vb^* \in \Z\setminus 0}|U^{d-1}(f_\vb(U^d(X, \vb, s)), \vb^*, s^*)|.\]
	We sum \eqref{eqn:two-U-intersection-n-dim} over all $\vb'\in B$ with $f_\vb(\vb') \ne f_\vb(\vzero)$. Combining it with \eqref{eqn:sum-over-intersection-n-dim-zero-part}, we have
	\begin{equation}\label{eqn:lem-subset-counting-sub-n-dim}\begin{split}
		\sum_{\vb'\in B}|U^d(X, \vb, s)\cap U^d(X, \vb', s)| & = \sum_{\vb'\in B: f_\vb(\vb') = f_\vb(\vzero)}|U^d(X, \vb, s)\cap U^d(X, \vb', s)| \\
		& \quad + \sum_{\vb'\in B: f_\vb(\vb') \ne f_\vb(\vzero)}|U^d(X, \vb, s)\cap U^d(X, \vb', s)|\\
		& \leq  \frac{4}{s\lambda_0}|U^d(X, \vb, s)| + \frac{s^*-1}{s}\sum_{\vb'\in B}|U^d(X, \vb', s)|\\
		& \quad + \frac{12}{s\lambda_0^2}\sum_{\vb^* \in \Z^{d-1}\setminus \{\vzero\}}|U^{d-1}(f_\vb(U^d(X, \vb, s)), \vb^*, s^*)|.
	\end{split}
	\end{equation}

	This establishes the desired inequality \eqref{eqn:sum-over-intersection-n-dim}. \end{proof}
	
	Lemma \ref{lem:sum-over-intersection-n-dim} is a useful bound for those $\vb$ for which $\lambda(\vb)$ (as defined in Lemma \ref{lem:single-intersection}) is not too small. The following lemma shows that there are not many choices of $\vb$ for which $\lambda(\vb)$ is small.

	\begin{lemma}\label{lem:number-theory-2}
	Let $d\geq 2$ be a positive integer, $n_1, \dots, n_d\in \N$ and $\epsilon \in (0, 1)$. If $\frac{1}{\epsilon} \leq n_i$ for all $1\leq i\leq d$, then there are at most $6^d\epsilon n_1\cdots n_d$ nonzero points $(b_1, b_2, \dots, b_d)\in [-n_1, n_1]\times \cdots \times [-n_d, n_d]$ with $|b_i/\gcd(b_1, \dots, b_d)| \leq \epsilon n_i$ for all $1\leq i\leq d$.
	\end{lemma}
	\begin{proof}
	For each $i$, the number of tuples with $b_i = 0$ is given by
	\[\prod_{j\ne i}(2n_j+1) \leq 3^dn_1\cdots n_d\cdot \frac{1}{n_i} \leq 3^d\epsilon n_1\cdots n_d.\]
	Hence the number of such tuples with at least one zero entry is at most $3^dd\epsilon n_1\cdots n_d$.
	
	We may next only consider the tuples with nonzero entries. Suppose that $\gcd(b_1, \dots, b_d) = k$. For any given $k$, we know that $b_i' = b_i/k$ satisfies that $|b_i'| \leq n_i/k$. From the problem condition, we further know that $|b_i'|\leq \epsilon n_i$. Hence when $k$ is fixed, the number of such tuples is at most $(2\epsilon)^d n_1\cdots n_d$ if $k\leq \frac{1}{\epsilon}$, and $(2/k)^dn_1\cdots n_d$ if $k > \frac{1}{\epsilon}$. Thus summing over $k$, we know that the number of such tuples $(b_1, b_2, \dots, b_d)$ is at most
	\[\sum_{1\leq k \leq \frac{1}{\epsilon}} 2^d\epsilon^dn_1\cdots n_d + \sum_{\frac{1}{\epsilon} < k} 2^d\frac{n_1\cdots n_d}{k^d} \leq 2^d\epsilon^{d-1} n_1\cdots n_d + 2^d\cdot 2\epsilon^{d-1} n_1\cdots n_d = 3\cdot 2^d\epsilon^{d-1} n_1\cdots n_d.\]
	Therefore the total number of such tuples is at most
	\[3^dd\epsilon n_1\cdots n_d + 3\cdot 2^d\epsilon^{d-1} n_1\cdots n_d \leq 6^d\epsilon n_1\cdots n_d. \]
	\end{proof}
 If $s\geq 2$ and $U^d(X, \vb, s)$ is nonempty, then $\vb\in [-\frac{N_1}{s-1}, \frac{N_1}{s-1}] \times \cdots \times [-\frac{N_d}{s-1}, \frac{N_d}{s-1}]$. In the lemma above we pick $n_i = \frac{N_i}{s-1}$. Then the lemma above gives an upper bound on the number of nonzero points $\vb$ whose $\lambda$ value (as defined in Lemma \ref{lem:single-intersection}) is at most $\frac{\epsilon}{s-1}$.

Finally, we need the following lemma.
	\begin{lemma}\label{lem:sum-of-intersection}
	Let $X$ be a set of size $m > 0$. Let $\{A_i\}_{i\in I}$ be a family of subsets of $X$ over indices $i\in I$. Then we have
	\[\sum_{i, j\in I}|A_i\cap A_j| \geq \frac{1}{m}\left(\sum_{i\in I}|A_i|\right)^2.\]
	\end{lemma}
	\begin{proof}
	We count the number of tuples in $T = \{(x, i, j)\in X\times I\times I: x\in A_i\cap A_j\}$. Note that if we fix $i$ and $j$, the number of choices for $x$ is exactly $|A_i\cap A_j|$. Hence we have $|T| = \sum_{i, j\in I}|A_i\cap A_j|.$

	For each $x\in X$, let $m_x = |\{i\in I: x\in A_i\}|$, the number of sets $A_i$ that contain $x$. First we see that $\sum_{x\in X}m_x = \sum_{i\in I}|A_i|$.
	Moreover, when counting $X$, once we fix $x\in X$, the number of choices for $(i, j)$ is $m_x^2$, so we have $ |T| = \sum_{x\in X}m_x^2.$
	By the Cauchy-Schwarz inequality, we have
	\[m\sum_{i, j\in I}|A_i\cap A_j| = m|T| = |X|\left(\sum_{x\in X}m_x^2\right) \geq \left(\sum_{x\in X}m_x\right)^2 = \left(\sum_{i\in I}|A_i|\right)^2.\]
	This gives the desired inequality.
	\end{proof}

	We next prove Lemma~\ref{lem:counting-subsets-3}. The proof is by induction on $d$. The following proposition handles the base case $d = 1$ and is due to 
	Matou\v sek and Spencer \cite{MS}.
	\begin{proposition}[Proposition 4.1 in \cite{MS}]\label{prop:MS-4.1}There exists an absolute constant $C$ such that the following holds. For positive integers $N$ and $m$, if $X\subseteq [N]$ is a subset of size $m$ and $s\geq 5\sqrt{m}$, then 
		\[\sum_{b\in \Z \setminus \{0\}}|U^1(X, b, s)| \leq C\frac{N^{\frac12}m^{\frac32}}{s}.\]
	\end{proposition}

	Now we have all the tools to set up the proof of Lemma~\ref{lem:counting-subsets-3}. We first describe the proof idea.
	Let us consider a fixed $d \geq 2$ with the induction hypothesis that the statement holds for $d-1$. Since we are to run an induction, the crux of the proof is to apply the induction hypothesis. Corollary~\ref{cor:projection} enables us to project the $d$-dimensional set $X\subseteq [N_1]\times \cdots \times [N_d]$ to a $(d-1)$-dimensional set $X^*\subseteq [N_1^*]\times \cdots \times [N_{d-1}^*]$ for some set $X^*$ and integers $N_j^*$ for $1\leq j\leq d-1$. Let $\rho = \frac{|X|}{N_1\cdots N_d}$ and $\rho^* = \frac{|X^*|}{N_1^*\cdots N_{d-1}^*}$ be the densities of the sets in the grids that they are subsets of. Fix $\epsilon = \rho^\gamma$ where $\gamma$ is chosen to be the exponent of $\rho$ in \eqref{eqn:counting-subsets-3}. As Lemma~\ref{lem:number-theory-2} says that only $O_d(\epsilon)$-fraction of $\vb$'s satisfy $\lambda(\vb) < \frac{\epsilon}{s}$, it allows us to only focus on $\vb$ with $\lambda(\vb)$ roughly $\frac{1}{s}$, off by a factor of at most $O_d(\epsilon) = O_d(\rho^\gamma)$. It follows that we can estimate both $N_1^*\cdots N_{d-1}^*$ and $X^* = f_\vb(U^d(X, \vb, s))$ within a factor of $\rho^{O_d(\gamma)}$ by applying Corollary~\ref{cor:projection} and  Lemma~\ref{lem:preimage-of-projection} respectively. Thus we can estimate $\rho^*$ within a factor of $\rho^{O_d(\gamma)}$. Finally we would like to apply Lemma~\ref{lem:sum-over-intersection-n-dim} and combine that with Lemma~\ref{lem:sum-of-intersection} to get the desired bound. Together these two lemmas give us an upper bound on the sum of the sizes of the $U^d(X,\vb,s)$ over all $\vb$ in a carefully chosen set $B\subseteq \mathbb{Z}^d$ in which the $i$-th coordinate is at most $N_i/(s-1)$ in absolute value for each $i$. 
	
	In the bound in Lemma~\ref{lem:sum-over-intersection-n-dim}, it is not hard to see that the first term on the right hand side is of lower order. By choosing $s^* = s\rho^\gamma$, we can save a factor of $\rho^\gamma$ in the second term. For the third term, we apply the induction hypothesis to save a factor of $(\rho^*)^{\gamma^*}$ (where the bound is expressed in terms of $|X^*|$, $N_1^*\cdots N_{d-1}^*$ and $\rho^*$). Since we can estimate each of them within a factor of $\rho^{O_d(\gamma)}$, we conclude that we save a factor of $\rho^{\gamma^* - O_d(\gamma)}$ in the third term. We can make $\gamma$ small enough so that $\gamma^* - O_d(\gamma) \geq \gamma$. In summary we save a factor of $O_d(\rho^\gamma)$ in all three terms, which is exactly what we need in Lemma~\ref{lem:counting-subsets-3}.

We next recall the statement of Lemma~\ref{lem:counting-subsets-3} and prove it.
	\counting*
	\begin{proof}[Proof of Lemma~\ref{lem:counting-subsets-3}]
	Let $C_0 > 1$ be an absolute constant that satisfies Proposition~\ref{prop:MS-4.1}. We prove by induction on $d$ that the statement holds for $C = C_0\cdot 5^d\cdot 2^{d^3}$. We may assume $N_1 = \min_{1\leq i\leq d}N_i$.
	
	If $m = 0$ or if $N_1 = \cdots = N_d = 1$, then the statement trivially holds. Hence we may assume that $m \geq 1$ and $N_1\cdots N_d > 1$. Therefore, noting that $\frac{1}{d+1} - \frac{\delta}{4^d(d+1)} > 0$, we have that $s$ satisfies
	\begin{equation}\label{eqn:lower-bound-s}
	    s \geq (N_1\cdots N_d)^{\frac{1}{d+1}}\rho^{\frac{\delta}{4^d(d+1)}} = (N_1\cdots N_d)^{\frac{1}{d+1} - \frac{\delta}{4^d(d+1)}}m^{\frac{\delta}{4^d(d+1)}} \geq (N_1\cdots N_d)^{\frac{1}{d+1} - \frac{\delta}{4^d(d+1)}}.
	\end{equation}
	It follows that $s\geq 2$. Also note that $\rho \leq 1$. We know that $2\leq s \leq \min_{1\leq i\leq d}N_i$. Hence we have
	\begin{equation}
	    \sum_{\vb\in \Z^d\setminus \{\vzero\}}|U^d(X, \vb, s)|\leq |X| \cdot \prod_{i=1}^d \left(4\frac{N_i}{s}+1\right) \leq m \cdot \prod_{i=1}^d 5\frac{N_i}{s} = 5^d\frac{mN_1\cdots N_d}{s^d}
	\end{equation}
	by Lemma~\ref{lem:counting-U-1}. Therefore the statement holds if $C\rho^{\frac{\min(\beta, \delta)}{4^d(d+2)!}} \ge 5^d$. Hence we may assume that $\rho^{\frac{\min(\beta, \delta)}{4^d(d+2)!}} < \frac{5^d}{C} = 2^{d^3}C_0 < 2^{d^3}$, or equivalently $\rho < 2^{-\frac{(d+2)!\cdot 4^dd^3}{\min(\beta, \delta)}} =: \rho_0$.
	
	We prove the base case $d = 1$ using Proposition~\ref{prop:MS-4.1}. Note that $\frac{\delta}{4^d(d+1)} = \frac{\delta}{8} \leq 1/8$ and $\frac{\min(\beta, \delta)}{4^d(d+2)!} \leq 1/24$. As we assumed that $\rho \leq \rho_0 = 2^{-\frac{24}{\min(\beta, \delta)}} < 5^{-\frac{8}{3}}$, we have $\sqrt{N_1}\rho^{\frac{\delta}{4^d(d+1)}} \geq \sqrt{N_1\rho} \rho^{-\frac{3}{8}}\geq 5\sqrt{m}.$
If $s\geq \sqrt{N_1}\rho^{\frac{\delta}{4^d(d+1)}} \geq 5\sqrt{m}$, then by Proposition~\ref{prop:MS-4.1} we have the desired inequality for $d = 1$
	\[\sum_{b\in \Z\setminus \{\vzero\}}|U^1(X, b, s)| \leq C_0\frac{mN_1}{s}\rho^\frac{1}{2} \leq C\frac{mN_1}{s}\rho^\frac{\min(\beta, \delta)}{4^d\cdot (d+2)!}.\]

	We next show the desired bound for $d \geq 2$, assuming the induction hypothesis for $d^* = d - 1$.
	Let $\gamma = \frac{\min(\beta, \delta)}{4^{d}\cdot (d+2)!}$ and $\epsilon = \rho^\gamma$. As $|U^d(X, \vb, s)|$ is zero if $\vb$ is not a nonzero integer point in $[-\frac{N_1}{s-1}, \frac{N_1}{s-1}]\times \cdots \times [-\frac{N_d}{s-1}, \frac{N_d}{s-1}]$, we may ignore these points in the summation. Let $B_0$ be the set of nonzero integer points in $[-\frac{N_1}{s-1}, \frac{N_1}{s-1}]\times \cdots \times [-\frac{N_d}{s-1}, \frac{N_d}{s-1}]$. Let $B_1$ be the set of points $\vb = (b_1, \dots, b_d)$ in $B_0$ for which $|b_i/\gcd(b_1, \dots, b_d)| \leq \epsilon \frac{N_i}{s-1}$ for all $1\leq i\leq d$. Let $B_2$ be the set of nonzero integer points $\vb$ in $B_0\setminus B_1$ for which $|U^d(X, \vb, s)| \leq \epsilon m$. Let $B = B_0\setminus (B_1\cup B_2)$. Therefore
	\begin{equation}\label{eqn:partition-sum}
	\sum_{\vb\in \Z^d\setminus \{\vzero\}} |U^d(X, \vb, s)| = \sum_{\vb\in B_1} |U^d(X, \vb, s)| + \sum_{\vb\in B_2} |U^d(X, \vb, s)| + \sum_{\vb\in B} |U^d(X, \vb, s)|.
	\end{equation}
	We estimate each term on the right hand side of \eqref{eqn:partition-sum}. As $s \leq N_{1}\rho^\beta \leq N_i\rho^\beta$, we have $n_i:= \frac{N_i}{s-1} \geq \rho^{-\beta} \geq \rho^{-\gamma} = \frac{1}{\epsilon}$. Hence we may apply Lemma~\ref{lem:number-theory-2} and conclude that
	\[|B_1| \leq 6^d \epsilon n_1\cdots n_d = 6^d\epsilon \frac{N_1\cdots N_d}{(s-1)^d} \leq 12^d\epsilon \frac{N_1\cdots N_d}{s^d} = 12^d \rho^\gamma \frac{N_1\cdots N_d}{s^d}.\]
	For each $\vb\in B_1$, we have $U^d(X, \vb, s) \subseteq X$, so $|U^d(X, \vb, s)| \leq m$. It follows that
	\begin{equation}\label{eqn:partition-sum-term-1}
	\sum_{\vb\in B_1} |U^d(X, \vb, s)|\leq |B_1|m \leq 12^d\rho^\gamma\frac{mN_1\cdots N_d}{s^d}.
	\end{equation}
	As $B_0\subseteq [-\frac{N_1}{s-1}, \frac{N_1}{s-1}]\times \cdots \times [-\frac{N_d}{s-1}, \frac{N_d}{s-1}]$ and $s\leq N_1\rho^\beta \leq N_i$, we have
	\[|B_0| \leq \prod_{i = 1}^d \left(2\frac{N_i}{s-1} + 1\right) \leq  \prod_{i = 1}^d \left(4\frac{N_i}{s} + 1\right) \leq 5^d\frac{N_1\cdots N_d}{s^d}.\]
	Therefore as $B_2\subseteq B_0$, we have that
	\begin{equation}\label{eqn:partition-sum-term-2}
	\sum_{\vb\in B_2} |U^d(X, \vb, s)| \leq |B_2| \cdot \epsilon m \leq 5^d\frac{N_1\cdots N_d}{s^d} \cdot \epsilon m = 5^d\rho^\gamma\frac{mN_1\cdots N_d}{s^d}.
	\end{equation}
Observe that $B\subseteq B_0$, and it follows
	\begin{equation}\label{eqn:bound-B-size}
	|B|\leq |B_0| \leq 5^d \frac{N_1\cdots N_d}{s^d}.
	\end{equation}
	Here \eqref{eqn:partition-sum-term-2} gives an upper bound on the second term in \eqref{eqn:partition-sum}.
	Finally we bound the third term. By Lemma~\ref{lem:sum-of-intersection}, noting that $\{U^d(X, \vb, s)\}_{\vb\in B}$ is a family of subsets of $X$, we have
	\begin{equation}\label{eqn:partition-sum-term-3-intersection}
	\sum_{\vb\in B}\sum_{\vb'\in B}|U^d(X, \vb, s)\cap U^d(X, \vb', s)| \geq \frac{1}{m}\left(\sum_{\vb\in B}|U^d(X, \vb, s)\right)^2.
	\end{equation}
We next give an upper bound on $\sum_{\vb'\in B}|U^d(X, \vb, s)\cap U^d(X, \vb', s)|$ for fixed $\vb\in B$ by Lemma~\ref{lem:sum-over-intersection-n-dim}. Before we can apply it, we need to make a few preparations to ensure that the conditions are satisfied.
	
	Since we have excluded elements in $B_1$, for any $\vb = (b_1, \dots, b_d) \in B$ there exists some index $i$ for which $|b_i/\gcd(b_1, \dots, b_d)| > \epsilon n_i = \epsilon \frac{N_i}{s-1}$. This implies that
	\begin{equation}\label{eqn:bound-lambda}
	\lambda_\vb := \max_{1\leq i\leq d} \frac{|b_i|}{\gcd(b_1, \dots, b_d)N_i} \geq \frac{\epsilon}{s-1} > \frac{\epsilon}{s}.
	\end{equation}
	Meanwhile, as $|b_i| \leq \frac{N_i}{s-1}$ for each $i$, we know that $\lambda_\vb \leq \frac{1}{s-1} \leq \frac{2}{s}$. Therefore $\lambda_\vb\in (\frac{\epsilon}{s}, \frac{2}{s}]$ for all $\vb\in B$. Hence $B$ satisfies the conditions in Lemma~\ref{lem:sum-over-intersection-n-dim} for $\lambda_0 := \frac{\epsilon}{s}$. 
	
	We fix an arbitrary $\vb\in B$. By Corollary~\ref{cor:projection} there exists a linear map $f_\vb: \Z^d\to \Z^{d-1}$ that satisfies conditions \eqref{item:nullspace-projection} and \eqref{item:range-projection}. Let $s^* = \lceil\epsilon s\rceil$. Now we know that $\vb$, $f_\vb$, and $s^*$ satisfy the conditions in Lemma~\ref{lem:sum-over-intersection-n-dim}. We apply it and get
	\begin{equation}\begin{split}\label{eqn:single-sum}
		\sum_{\vb'\in B}|U^d(X, \vb, s)\cap U^d(X, \vb', s)| & \leq  \frac{4|U^d(X, \vb, s)|}{s\lambda_0} + \frac{s^*-1}{s}\sum_{\vb'\in B}|U^d(X, \vb', s)| \\
		& \quad  + \frac{12}{s\lambda_0^2}\sum_{\vb^*\in \Z^{d-1}\setminus \{\vzero\}}|U^{d-1}(f_\vb(U^d(X, \vb, s)), \vb^*, s^*)|.
	\end{split}
	\end{equation}
	For the third term on the right hand side of \eqref{eqn:single-sum}, we would like to apply the induction hypothesis. 
To do this, we need to verify the various conditions in the statement by proving following claims. 
	
	As $f_\vb$ satisfies condition \eqref{item:range-projection} in Corollary~\ref{cor:projection}, there exist positive integers $N_1^*, \dots, N_{d-1}^*$ such that $f_\vb([N_1]\times \cdots \times [N_d]) \subseteq [N_1^*]\times \cdots \times [N_{d-1}^*]$, 
	\begin{equation}\label{eqn:bound-product-N^*-as-lambda}
	\frac{1}{2}\lambda_\vb N_1\cdots N_d \leq N_1^* \cdots N_{d-1}^* \leq 2^{d^2}\lambda_\vb N_1\cdots N_d,
	\end{equation}
  Let $M := \min_{1\leq i\leq d-1}N_i^*  \geq N_1$. As $\lambda_\vb\in (\frac{\epsilon}{s}, \frac{2}{s}]$, we have
	\begin{equation}\label{eqn:bound-product-N^*}
	\frac{\epsilon}{2s} N_1\cdots N_d \leq N_1^* \cdots N_{d-1}^* \leq \frac{2^{d^2+1}}{s}N_1\cdots N_d.
	\end{equation}
	Let $X^* = f_\vb(U^d(X, \vb, s))$. As $f_\vb$ satisfies \eqref{item:range-projection} in Corollary~\ref{cor:projection}, we have $X^* \subseteq [N_1^*]\times \cdots \times [N_{d-1}^*]$. Let $m^* := |X^*|$ and $\rho^* := m^*/(N_1^*\cdots N_{d-1}^*)$.
	
	The following claims allow us to apply the induction hypothesis to $(N_j^*)_{j=1}^{d-1}$, $X^*$, and $s^*$.
	\begin{claim}\label{claim:condition-N^*}
	$N_1^*\cdots N_{d-1}^* \leq M^{d-\delta/3}.$
	\end{claim}
	\begin{proof}[Proof of Claim~\ref{claim:condition-N^*}]

	Note that $N_1\leq M$, so we have that 
	\begin{equation}\label{eqn:upper-bound-N^*-product}N_1\cdots N_d\leq N_1^{d+1-\delta} \leq M^{d+1-\delta}.\end{equation}
	Since $m\geq 1$, we know that $M^{d+1} \geq N_1\cdots N_d \geq \frac{1}{\rho} > 2^\frac{(d+2)!\cdot 4^dd^3}{\min(\beta, \delta)}$ and so $M \geq 2^{\frac{16(d+2)d^3}{\delta}} > 2^{\frac{12(d^2+1)}{\delta}}$.
	Combining \eqref{eqn:upper-bound-N^*-product} with \eqref{eqn:lower-bound-s} and \eqref{eqn:bound-product-N^*}, we have
	\begin{equation}\label{eqn:upper-bound-N^*-product-by-N^*}
	    N_1^*\cdots N_{d-1}^* \leq 2^{d^2}\lambda_\vb N_1\cdots N_d \leq 2^{d^2+1} (N_1\cdots N_d)^{\frac{d}{d+1}+\frac{\delta}{4^d(d+1)}} \leq 2^{d^2+1}M^{\frac{d+1-\delta}{d+1}(d+\frac{\delta}{4^d})}.
	\end{equation}
	For the exponent of $M$ on the right hand side in \eqref{eqn:upper-bound-N^*-product-by-N^*}, we know that
	\[\frac{d+1-\delta}{d+1}\left(d+\frac{\delta}{4^d}\right) = d+\frac{\delta}{4^d} - \delta\frac{d+\frac{\delta}{4^d}}{d+1} < d + \frac{\delta}{4} - \frac{\delta}{2} = d - \frac{\delta}{3} - \frac{\delta}{12}.\]
	Therefore, we can simplify \eqref{eqn:upper-bound-N^*-product-by-N^*} and get
	\[N_1^*\cdots N_{d-1}^* \leq 2^{d^2+1}M^{d - \frac{\delta}{3} - \frac{\delta}{12}} \leq M^{d - \frac{\delta}{3}}\]
	as expected, where in the last inequality we use that $M \geq 2^{\frac{12(d^2+1)}{\delta}}$.
	\end{proof}

	By Claim~\ref{claim:condition-N^*}, we can apply the induction hypothesis to $X^* \subseteq [N_1^*]\times \cdots\times [N_{d-1}^*]$ and $\delta^* = \delta/3$. It remains to show that $s^*$ is also in the desired range.
	
	\begin{claim}\label{claim:condition-s^*}$(N_1^*\cdots N_{d-1}^*)^\frac{1}{d} (\rho^*)^{\frac{\delta/3}{4^{d-1}d}} \leq s^* \leq M(\rho^*)^{\beta}.$
	\end{claim}
	\begin{proof}[Proof of Claim~\ref{claim:condition-s^*}]
	By Lemma~\ref{lem:preimage-of-projection}, we have	\begin{equation}\label{eqn:bound-m^*} \frac{\lambda_\vb}{2}|U^d(X, \vb, s)| \leq m^* \leq \frac{1}{s}|U^d(X, \vb, s)|.\end{equation}
	Note that $\epsilon = \rho^{\gamma}$ and that $\rho = \frac{m}{N_1\cdots N_d}$. As $|U^d(X, \vb, s)| \leq m$, combining with \eqref{eqn:bound-product-N^*}, we have
	\begin{equation}\label{eqn:upper-bound-rho^*}\rho^* = \frac{m^*}{N_1^*\cdots N_{d-1}^*} \leq \frac{\frac{m}{s}}{\frac{\epsilon}{2s}N_1\cdots N_d} = \frac{2}{\epsilon}\frac{m}{N_1\cdots N_d} = 2\rho^{1-\gamma}.\end{equation}
	Recall that $\gamma = \frac{\min(\beta, \delta)}{4^d\cdot (d+2)!} \leq \frac{\delta}{12\cdot 4^dd}$, and so
	\[\frac{\delta}{3\cdot 4^{d-1}d}(1-\gamma) > \frac{\delta}{3\cdot 4^{d-1}d}-\gamma =  \frac{\delta}{4^dd}+ \frac{\delta}{3\cdot 4^{d}d} - \gamma \geq\frac{d+1}{d}\cdot \frac{\delta}{4^d(d+1)} + \gamma + \frac{\delta}{6\cdot 4^dd}.\]
	As a result, raising both sides of $(N_1\cdots N_d)^{\frac{1}{d+1}}\rho^{\frac{\delta}{4^d(d+1)}} \leq s$ to the $\frac{d+1}{d}$-th power, we have
	\[s\rho^\gamma \geq s^{-\frac{1}{d}} \cdot (N_1\cdots N_d)^{\frac{1}{d}} \rho^{\frac{d+1}{d}\frac{\delta}{4^d(d+1)}+\gamma} \geq 2^{-\frac{d^2+1}{d}-\frac{\delta}{3\cdot 4^{d-1}d}}\rho^{-\frac{\delta}{6\cdot 4^dd}}\left(\frac{2^{d^2+1}}{s}N_1\cdots N_d\right)^\frac{1}{d} (2\rho^{1-\gamma})^{\frac{\delta}{3\cdot 4^{d-1}d}}.\]
	Note that $s^* \geq \epsilon s = \rho^\gamma s$. By \eqref{eqn:upper-bound-rho^*}, we have $\rho^* \leq 2\rho^{1-\gamma}$. Combining these with \eqref{eqn:bound-product-N^*}, we have
	\begin{equation}\label{eqn:claim-s^*-1}s^* \geq s\rho^\gamma \geq 2^{-\frac{d^2+1}{d}-\frac{\delta}{3\cdot 4^{d-1}d}}\rho^{-\frac{\delta}{6\cdot 4^dd}}(N_1^*\cdots N_{d-1}^*)^{\frac{1}{d}}(\rho^*)^{\frac{\delta}{3\cdot 4^{d-1}d}} \geq (N_1^*\cdots N_{d-1}^*)^{\frac{1}{d}}(\rho^*)^{\frac{\delta}{3\cdot 4^{d-1}d}}.\end{equation}
	Here in the last inequality we use that $\rho^{-\frac{\delta}{6\cdot 4^dd}} \geq \rho^{-2\gamma} \geq 2^{2d^3} \geq 2^{\frac{d^2+1}{d} + \frac{\delta}{3\cdot 4^{d-1}d}}$.
		Note that from our choice of $\vb\in B$, $|U^d(X, \vb, s)|$ is at least $\epsilon m$. Therefore combining with \eqref{eqn:bound-m^*} we have $m^*\geq \frac{\epsilon}{2}\lambda_\vb m$. Combining this with \eqref{eqn:bound-product-N^*-as-lambda} and $\epsilon = \rho^{\gamma}$, we have
	\begin{equation}\label{eqn:lower-bound-rho^*}
	\rho^* = \frac{m^*}{N_1^*\cdots N_{d-1}^*} \geq \frac{\frac{\epsilon}{2}\lambda_\vb m}{2^{d^2}\lambda_\vb N_1\cdots N_d} \geq 2^{-d^2-1}\cdot \epsilon\rho = 2^{-d^2-1}\rho^{1+\gamma}.
	\end{equation}
	By \eqref{eqn:lower-bound-s}, we know that $\epsilon s \geq \rho^\gamma \cdot (m/\rho)^{\frac{1}{d+1}-\frac{\delta}{4^d(d+1)}} \geq \rho^{-\frac{1}{d+1}+\frac{\delta}{4^d(d+1)} + \gamma}$. As $\frac{\delta}{4^d(d+1)}+\gamma < \frac{1}{d+1}$ and $\rho < 1$, we have $\epsilon s > 1$, so $s^* = \lceil\epsilon s\rceil\leq 2\epsilon s$. Note that $s \leq N_1\rho^\beta \leq M\rho^\beta$. Therefore, we have
	\begin{equation}\label{eqn:claim-s^*-2}
	    s^* \leq 2\epsilon s \leq 2\rho^\gamma M\rho^{\beta} = 2^{1+\beta(d^2+1)}\rho^{\gamma - \beta\gamma} M (2^{-d^2-1}\rho^{1+\gamma})^\beta \leq M (\rho^{*})^\beta,
	\end{equation}
	where in the last inequality we use \eqref{eqn:lower-bound-rho^*} and that $\rho^{-\gamma + \beta\gamma} \geq \rho^{\gamma/2} > 2^{d^3/2} > 2^{1+\beta(d^2+1)}$ for $d\geq 2$.
	\end{proof}	
	
	We conclude that the conditions for the induction hypothesis are satisfied by $d^* = d-1$, $(N_i^*)_{i=1}^{d-1}$, $\delta^* = \delta/3$, $\beta^* = \beta$, $X^* = f_\vb(U^d(X, \vb, s))$, and $s^* = \lceil \epsilon s\rceil$. Applying the induction hypothesis we get
	\begin{equation}\label{eqn:apply-induction-1}
	    \sum_{\vb^*\in \Z^{d-1}\setminus \{\vzero\}}|U^{d-1}(f_\vb(U^d(X, \vb, s)), \vb^*, s^*)| \leq C^*\frac{N_1^*\cdots N_{d-1}^*m^*}{(s^*)^{d-1}}(\rho^*)^{\frac{\min(\beta, \delta/3)}{4^{d-1}(d+1)!}},
	\end{equation}
	where $C^* = C_0\cdot 5^{d-1}\cdot 2^{(d-1)^3}$. Using \eqref{eqn:bound-product-N^*}, \eqref{eqn:bound-m^*}, \eqref{eqn:upper-bound-rho^*}, and that $s^* \geq \rho^\gamma s$, we have
	\begin{equation}\label{eqn:apply-induction-2}C^*\frac{N_1^*\cdots N_{d-1}^*m^*}{(s^*)^{d-1}}(\rho^*)^{\gamma^*} \leq C^* \frac{2^{d^2+1}N_1\cdots N_d}{\rho^{(d-1)\gamma}s^{d+1}}\rho^{(1-\gamma)\gamma^*}|U^d(X, \vb, s)|\end{equation}
	where for simplicity we denote $\gamma^* := \frac{\min(\beta, \delta/3)}{4^{d-1}\cdot (d+1)!} \geq \frac{\min(\beta, \delta)}{3\cdot 4^{d-1}\cdot (d+1)!} = \frac{4(d+2)}{3}\gamma$. Note that exponent of $\rho$ on the right hand side of \eqref{eqn:apply-induction-2} satisfies $(1-\gamma)\gamma^* - (d-1)\gamma = \gamma^* - (d-1+\gamma^*)\gamma > 3\gamma$. Combining this with the inequalities \eqref{eqn:apply-induction-1} and \eqref{eqn:apply-induction-2}, we have
	\begin{equation}
	\sum_{\vb^*\in \Z^{d-1}\setminus \{\vzero\}}|U^{d-1}(f_\vb(U^d(X, \vb, s)), \vb^*, s^*)| \leq 2^{d^2+1}C^* \rho^{3\gamma}\frac{N_1\cdots N_d}{s^{d+1}}|U^d(X, \vb, s)|.
	\end{equation}
Put this into \eqref{eqn:single-sum}. Note that $\lambda_0 = \frac{\epsilon}{s} = \rho^\gamma / s$ and that $\frac{s^*-1}{s}\leq \frac{\epsilon s}{s} = \rho^\gamma$. We have
	\begin{equation}\label{eqn:single-sum-simplified}\begin{split}
		\sum_{\vb'\in B}|U^d(X, \vb, s)\cap U^d(X, \vb', s)| & \leq  4\rho^{-\gamma}|U^d(X, \vb, s)| + \rho^\gamma\sum_{\vb'\in B}|U^d(X, \vb', s)| \\
		&  + 12\cdot 2^{d^2+1}C^* \rho^{\gamma}\frac{N_1\cdots N_d}{s^{d}}|U^d(X, \vb, s)|.
	\end{split}
	\end{equation}
 We sum \eqref{eqn:single-sum-simplified} over $\vb\in B$. 
 Recall that $\gamma = \frac{\min(\beta, \delta)}{4^d\cdot (d+2)!} < \frac{\beta}{4}$. We have
	\[\frac{N_1\cdots N_d}{s^d} \geq \frac{N_1^d}{s^d} \geq \rho^{-d\beta} > \rho^{-2\gamma},\]
	and it follows that $\rho^{-\gamma} < \rho^\gamma\frac{N_1\cdots N_d}{s^d}$. Using this and \eqref{eqn:bound-B-size}, we have
	\[\begin{split}\sum_{\vb, \vb'\in B}|U^d(X, \vb, s)\cap U^d(X, \vb', s)| & \leq \left(4\rho^{-\gamma} + \rho^\gamma |B| + 12\cdot 2^{d^2+1}C^*\rho^\gamma \frac{N_1\cdots N_d}{s^d}\right)\sum_{\vb\in B}|U^d(X, \vb, s)|\\
	& \leq \left(4 + 5^d + 12\cdot 2^{d^2+1}C^*\right)\rho^\gamma \frac{N_1\cdots N_d}{s^d}\sum_{\vb\in B}|U^d(X, \vb, s)|.\end{split}\]
	Combine this with \eqref{eqn:partition-sum-term-3-intersection}, we have
	\begin{equation}\label{eqn:partition-sum-term-3}
	\sum_{\vb\in B}|U^d(X, \vb, s)| \leq \left(4 + 5^d + 12\cdot 2^{d^2+1}C^*\right)\rho^\gamma \frac{N_1\cdots N_d}{s^d}m.
	\end{equation}
Substituting in the bounds \eqref{eqn:partition-sum-term-1}, \eqref{eqn:partition-sum-term-2}, and \eqref{eqn:partition-sum-term-3} into \eqref{eqn:partition-sum}, we have
	\[\sum_{\vb\in \Z^d\setminus \{\vzero\}} |U^d(X, \vb, s)| \leq \left(12^d + 5^d + 4 + 5^d + 12\cdot 2^{d^2+1}C^*\right)\rho^\gamma \frac{mN_1\cdots N_d}{s^d}.\]
	Finally, we show that the sum of the additive constant above is at most $C$.
	Recall that $C^* = C_0\cdot 5^{d-1}\cdot 2^{(d-1)^3}$ and $C = C_0\cdot 5^d\cdot 2^{d^3}$, we know that $24\cdot 2^{d^2}C^* \leq \frac{3}{5}\cdot 5\cdot 2^{3d^2-3d+1}C^* = \frac{3C}{5}$  as $d\geq 2$. For other terms, we have $4\leq 5^d \leq 12^d \leq 2^{d^3} = \frac{C}{5^dC_0} \leq \frac{C}{25}$ and it follows that $12^d + 5^d + 4 + 5^d + 12\cdot 2^{d^2+1}C^* \leq \frac{4}{25} C + \frac{3}{5}C < C$. Thus the statement holds for $d$. Therefore we conclude that the statement holds for all positive integer $d$.\end{proof}

	\section{A proof of the lower bound in Theorem~\ref{thm:rectangular}}\label{sec: rec-lower}
    We prove the following result, which is the lower bound on the discrepancy in Theorem~\ref{thm:rectangular}.
	\begin{theorem}\label{thm:rectangular-lower}
	    For any positive integer $d$, there exists a constant $c_d>0$ such that the following holds. For positive integer $N_1,\ldots,N_d$, letting $\vN = (N_1, \dots, N_d)$, we have
	    \[\disc(\mathcal{A}_{\vN}) \geq c_d\max_{I\subseteq [d]}\left(\prod_{i\in I}N_i\right)^\frac{1}{2|I|+2}.\]
	    Here by convention if $I = \emptyset$ then $\prod_{i\in I}N_i = 1$.
	\end{theorem}
	
	Roth \cite{Roth} proved the case $d=1$, and Valk\'o \cite{Valko} proved the case that the $N_i$'s are equal. Similar to these previous results, our proof uses Fourier analysis. We first set up some notations. Let $f, g: \Z^d \to \C$ be two functions that each has finite support. The Fourier transform $\widehat{f}: [0, 1]^d\to \C$ is given by $\widehat{f}(\vr) = \sum_{\vx\in \Z^d}f(\vx)e^{-2\pi i\vx\cdot \vr}$. The convolution $f*g: \Z^d\to \C$ is given by $f*g(\vx) = \sum_{\vr\in \Z^d}f(\vr)g(\vx-\vr)$, which also has finite support.	With these notations, we have the convolution identity $\widehat{f*g} = \widehat{f}\cdot \widehat{g}$ and Parseval's identity
	\[\sum_{\vx\in \Z^d}f(\vx)\overline{g(\vx)} = \int_{[0, 1]^d} \widehat{f}(\vr) \overline{\widehat{g}(\vr)}\,d\vr.\]
	
	In the proof of Theorem~\ref{thm:rectangular-lower} below, for a vector $\vx \in \mathbb{Z}^d$, we let $\vx_i$ denote the $i^{\textrm{th}}$ coordinate of $\vx$. 
	
	\begin{proof}[Proof of Theorem~\ref{thm:rectangular-lower}]
	We take $c_d = \frac{6^{-d/2}}{2}$ Let $\Omega = [N_1]\times \cdots \times [N_d]\subseteq \Z^d$. Fix any $\chi: \Omega \to \{1, -1\}$. Let $T = \max_{A\in \mathcal{A}_\vN}|\chi(A)|$. It suffices to show that $T \geq c_d\max_{I\subseteq [d]}\left(\prod_{i\in S}N_i\right)^\frac{1}{2|I|+2}$.
	
	For $\chi: \Omega\to \{1, -1\}$, we may extend it to a function $\chi: \Z^d \to \{-1, 0, 1\}$ by assigning $0$ to $\Z^d\setminus \Omega$. Clearly $\chi$ takes nonzero values on $N_1\cdots N_d$ points. Hence we may apply Parseval's identity and get
	\begin{equation}\label{eqn:chi-norm}
	    \int_{[0, 1]^d} \widehat{\chi}(\vr) \overline{\widehat{\chi}(\vr)}\,d\vr = \sum_{\vx\in \Z^d}\chi(\vx)\overline{\chi(\vx)} = N_1\cdots N_d.
	\end{equation}
	
	Let $L$ be a positive integer and $D_1, \dots D_d$ be nonnegative integers to be determined later. For each $\vb\in \Z^d\setminus \vzero$ satisfying that $\vb_i\in [-D_i, D_i]$ for $1\leq i\leq d$, let $g_\vb: \Z^d\to \C$ be the indicator function of the set $\{0, \vb, \dots, (L-1)\vb\}$. Now we have for each $\vx\in \Z^d$, 
	\[g_\vb*\chi(\vx) = \sum_{t=0}^{L-1}\chi(\vx - t\vb) = \chi(\Omega\cap \{\vx - t\vb: 0\leq t < L\}).\]
	Since $\Omega\cap \{\vx - t\vb: 0\leq t < L\}$ is an arithmetic progression contained in $\Omega$, it is a set in $\mathcal{A}_\vN$, so $|g_\vb*\chi(\vx)| \leq T$. This is true for all $\vx\in \Z^d$. Also note that $|g_\vb*\chi(\vx)|$ is nonzero only when $\vx - t\vb\in \Omega$ for some $0\leq t < L$. In this case we have $\vx_i \in [1-LD_i, N_i+LD_i]$ for each $1\leq i\leq d$. Therefore, $g_\vb*\chi$ is nonzero on at most $\prod_{i=1}^d(N_i+2LD_i)$ points in $\Z^d$. We have
	\begin{equation}\label{eqn:upper-L2}
	\sum_{\vx\in \Z^d}{g_\vb * \chi}(\vx) \overline{{g_\vb * \chi}(\vx)} = \sum_{\vx\in \Z^d}\left|{g_\vb * \chi}(\vx)\right|^2 \leq T^2\prod_{i=1}^d(N_i+2LD_i).\end{equation}
	By the convolution identity and Parseval's identity, we have
	\begin{equation}\label{eqn:upper-Parseval}
	    \sum_{\vx\in \Z^d}{g_\vb * \chi}(\vx) \overline{{g_\vb * \chi}(\vx)} = \int_{[0, 1]^d}\widehat{g_\vb * \chi}(\vr) \overline{\widehat{g_\vb * \chi}(\vr)}\,d\vr = \int_{[0, 1]^d}\widehat{g_\vb}(\vr)\widehat{\chi}(\vr) \overline{\widehat{g_\vb}(\vr)\widehat{\chi}(\vr)}\,d\vr
	\end{equation}
	Combining \eqref{eqn:upper-L2} and \eqref{eqn:upper-Parseval}, we get that for any nonzero $\vb\in [-D_1, D_1] \times \dots \times [-D_d, D_d]$
	\begin{equation}\label{eqn:single-b}
	    \int_{[0, 1]^d}\left|\widehat{g_\vb}(\vr)\right|^2\widehat{\chi}(\vr) \overline{\widehat{\chi}(\vr)}\,d\vr \leq T^2 \prod_{i=1}^d(N_i + 2LD_i).
	\end{equation}
	Let $A$ be the set of integer points in $[0, D_1] \times \cdots \times [0, D_d]$ and $B$ be the set of nonzero integer points in $[-D_1, D_1] \times \cdots \times [-D_d, D_d]$. Clearly any two distinct points in $A$ have their difference in $B$. The number of points in $B$ is at most $\prod_{i=1}^d(2D_i+1)$. Hence if we sum over $\vb\in B$ in \eqref{eqn:single-b}, we get
	\begin{equation}\label{eqn:sum-b}
	    \int_{[0, 1]^d}\left(\sum_{\vb\in B}\left|\widehat{g_\vb}(\vr)\right|^2\right)\widehat{\chi}(\vr) \overline{\widehat{\chi}(\vr)}\,d\vr \leq T^2 \prod_{i=1}^d(N_i + 2LD_i)(2D_i+1)
	\end{equation}
	Fix any $\vr\in [0, 1]^d$. By the pigeonhole principle, there exists two distinct $\va, \va'\in A$ such that the fractional parts of $\va\cdot \vr$ and $\va'\cdot \vr$ differ by at most $1/|A|$. Hence for any $\vr$ we can find $\vb'\in B$ (we shall take $\vb' = \va-\va'$ or $\vb' = \va'-\va$) such that the fractional part of $\vb'\cdot \vr$ is in $[0, 1/|A|]$. If $L \leq \frac{|A|}{2} = \frac{1}{2}\prod_{i=1}^d (D_i+1)$, then for any $\vr\in [0, 1]^d$,
	\[\sum_{\vb\in B}\left|\widehat{g_\vb}(\vr)\right|^2 \geq \left|\widehat{g_{\vb'}}(\vr)\right|^2 = \left|\sum_{t=0}^{L-1}e^{-t2\pi i\vb'\cdot \vr}\right|^2 \geq \frac{4}{\pi^2}L^2.\]
	Put this into \eqref{eqn:sum-b} and combine with \eqref{eqn:chi-norm}. We conclude that for any positive integer $L$ and nonnegative integers $D_1, \dots, D_d$ such that $L\leq \frac{(D_1+1)\cdots (D_d+1)}{2}$, then
	\begin{equation}\label{eqn:Fourier-lower-unsimplified}
	    T^2\prod_{i=1}^d(N_i+2LD_i)(2D_i+1) \geq \frac{4}{\pi^2}L^2\prod_{i=1}^dN_i.
	\end{equation}
	
	Let $R = \max_{I\subseteq [d]}\left(\prod_{i\in I}N_i\right)^\frac{1}{2|I|+2}$. If $R \leq 2$, the statement is trivial as $c_d = \frac{6^{-d/2}}{2} \leq 1/2$. Therefore we may assume that $R > 2$ and the maximum in the definition of $R$ is achieved by some nonempty $I\subseteq [d]$. For each $j\in I$, we have $R \geq \left(\prod_{i\in I\setminus \{j\}}N_i\right)^\frac{1}{2|I|}$, so $N_j \geq R^2$. With these properties, we may now choose the values of $L$ and $D_1, \dots, D_d$. We set $L = \lfloor R^2/2\rfloor$, $D_{i} = \lfloor \frac{N_{i}}{R^2}\rfloor$ for $i\in I$, and $D_j = 0$ for each $j\notin I$. 
Since
	\[\frac{(D_1+1)\cdots (D_d+1)}{2} \geq  \frac{\prod_{i\in I}N_i}{2R^{2|I|}} = \frac{R^2}{2} \geq L,\]
	we can apply \eqref{eqn:Fourier-lower-unsimplified} to these variables. For $j\notin I$, as $D_j = 0$, we have $(N_j+2LD_j)(2D_j+1) = N_j$. For $i\in I$, since $N_i \geq R^2$, we have $N_i/R^2 \geq D_i \geq 1$, so
	\[(N_i+2LD_i)(2D_i+1) \leq \left(N_i + 2\cdot \frac{R^2}{2} \cdot \frac{N_i}{R^2}\right)\cdot 3\frac{N_i}{R^2} = 6\frac{N_i^2}{R^2}.\]
	Put these into \eqref{eqn:Fourier-lower-unsimplified}. Note that $L \geq \frac{R^2}{2}$. We have
	\[T^2 \prod_{i\in I}6\frac{N_i^2}{R^2} \cdot \prod_{j\notin I}N_j \geq \frac{R^4}{4}\prod_{i=1}^dN_i.\]
	Also note that $\prod_{i\in I}N_i = R^{2|I|+2}$. We conclude that $T \geq \frac{6^{-\frac{|I|}{2}}}{2}R \geq c_dR$.
	\end{proof}
	
	\section{A proof of the upper bound in Theorem~\ref{thm:rectangular}}\label{sec: rec-upper}

In this section, we aim to generalize the upper bound in Theorem~\ref{thm:almost-cubes} to all grids of differing side lengths. The following lemma allows us to remove dimensions of short side lengths.
\begin{lemma}\label{lem:slice}
Let $d\geq 2$ be a positive integer and $N_1, \dots, N_d$ be positive integers. Then for $\vN = (N_1, \dots, N_d)$ and $\vN' = (N_1, \dots, N_{d-1})$, we have
\[\disc(\mathcal A_{\vN}) \leq \max\left(\disc(\mathcal A_{\vN'}), \sqrt{6N_d\log(2N_1\cdots N_d)}\right).\]
\end{lemma}
\begin{proof}
Firstly we may choose an optimal coloring $\chi'$ for the grid $[N_1]\times \cdots \times [N_{d-1}]$ that achieves discrepancy $\disc(\mathcal A_{\vN'})$.

We extend this coloring to a coloring $\chi: [N_1]\times \cdots \times [N_d] \to \{1, -1\}$ by the following procedure. Take $N_d$ i.i.d. Rademacher random variables $v(i)$ for $1\leq i\leq N_d$ (i.e. $\Pr(v(i) = 1) = \Pr(v(i) = -1) = \frac{1}{2}$). Now we set $\chi(x_1, \dots, x_d) = \chi'(x_1, \dots, x_{d-1})v(x_d)$ for any $(x_1, \dots, x_d)\in [N_1]\times \cdots \times [N_d]$. 

Now we analyze $\chi(S)$ for $S\in \mathcal A_\vN$. Let $(k_1, \dots, k_d)$ be the common difference of arithmetic progression $S$. If $k_d = 0$, then all elements in $S$ share the same $d$-th coordinate $x_d$, so we may write $S$ as $S'\times \{x_d\}$, where $S'$ is also an arithmetic progression with common difference $(k_1, \dots, k_{d-1})$. By our construction of $\chi$, we have $|\chi(S)| = |\chi'(S')v(x_d)| \leq \disc(\mathcal A_{\vN'})$.

Otherwise if $k_d\ne 0$, then all elements in $S$ have distinct $d$-th coordinates, and $|S|\leq N_d$. Since $\chi'$ is deterministic, we know that $\chi(S)$ is a summation of $|S|$ i.i.d. Rademacher random variables. Now by the Chernoff bound (e.g. see Theorem A.1.1 and Corollary A.1.2 in \cite{AS}), we have
\[\Pr(|\chi(S)| > \sqrt{6N_d\log(2N_1\cdots N_d)}) \leq 2e^{-\frac{6N_d\log(2N_1\cdots N_d)}{2|S|}} \leq \frac{1}{4}(N_1\cdots N_d)^{-3} < (N_1\cdots N_d)^{-3},\]
where in the last inequality we use that $|S| \leq N_d$. Finally we apply the union bound on all $S$ with $k_d\ne 0$. Clearly there are $N_1\cdots N_d$ ways to pick the first element in the arithmetic progression, and at most $N_1\cdots N_d$ ways to pick the last element, and at most $N_d$ ways to choose $|S|$ (as $1\leq |S| \leq N_d$). Once these are chosen, then clearly $S$ is determined as the common difference in the last coordinate is determined. Hence the total number of distinct $S$ in $\mathcal A_\vN$ with $k_d\ne 0$ is at most $(N_1\cdots N_d)^3$. By union bound, we conclude that there exists a choice of $v$ such that $|\chi(S)|\leq \sqrt{6N_d\log(2N_1\cdots N_d)}$ for all $S\in \mathcal{A}_{\vN}$ with distinct $d$-th coordinates.

In summary, we conclude that there is a choice of $\chi:[N_1]\times \cdots \times[N_d] \to \{1,-1\}$ so that
\[\max_{S\in \mathcal{A}_\vN}|\chi(S)| \leq \max\left(\disc(\mathcal A_{\vN'}), \sqrt{6N_d\log(2N_1\cdots N_d)}\right).\]
Note that $\disc(\mathcal{A}_\vN)$ is defined as the minimum over all $\chi$, so we have the desired inequality.
\end{proof}

\begin{proof}[Proof of the upper bound in Theorem~\ref{thm:rectangular}]
    Suppose that $N_1 \geq N_2 \geq \cdots \geq N_d \geq 1 =: N_{d+1}$. Assume that $N_1$ is sufficiently large to avoid triviality.
    
    Let $R_i = \left(\prod_{j=1}^i N_j\right)^{\frac{1}{i+1}}$ for $1\leq i\leq d$. Clearly 
    \[\max_{I\subseteq [d]}\left(\prod_{i\in I}N_i\right)^\frac{1}{|I|+1} = \max_{1\leq i\leq d} R_i.\]
    Now we take $t$ to be the first index $1\leq i\leq d$ such that $R_i > \frac{N_{i+1}}{(\log(N_1\cdots N_d))^{\frac{1}{2}}}$. By repeatedly applying Lemma~\ref{lem:slice}, for $\vN' = (N_1, \dots, N_t)$, we have
    \begin{equation}\label{eqn: upper disc}
        \disc(\mathcal{A}_\vN) \leq \max\left(\disc(\mathcal{A}_{\vN'}), 4\sqrt{N_{t+1}\log(N_1\cdots N_d)}\right).
    \end{equation}
    By our choice of $t$, we have $4\sqrt{N_{t+1}\log(N_1\cdots N_d)} \leq 4\sqrt{R_t}(\log (N_1\cdots N_d))^{\frac{3}{4}}$.
    
    Also we have $N_t \geq R_{t-1}(\log(N_1\cdots N_d))^{\frac{1}{2}}$, so $N_t \geq R_t(\log(N_1\cdots N_d))^{\frac{t}{2(t+1)}} \geq R_t(\log(N_1\cdots N_d))^{\frac{1}{4}}$. Consequently, we may pick $\delta = O\left(\frac{\log (N_1\cdots N_d)}{\log\log(N_1\cdots N_d)}\right)$ so that $N_t^{t+1-\delta} = R_t^{t+1}$. By Theorem~\ref{thm:almost-cubes} we have
    \[\disc(\mathcal A_{\vN'}) = O_d\left(\sqrt{R_t}\frac{\log (N_1\cdots N_d)}{\log\log (N_1\cdots N_d)}\right).\]
This completes the proof by invoking \eqref{eqn: upper disc}.    
\end{proof}
\begin{remark}
    The above proof gives that we can take $C_d = 2^{O(d^3)}$ in Theorem~\ref{thm:rectangular}.
\end{remark}
	
\section{Concluding remarks }\label{sec: conclusion and open problem}

Theorem~\ref{thm:rectangular} determines $\disc(\CA_\vN)$ up to a constant factor for many $\vN$'s. However, 
even when $d=2$, there is a regime where the upper and lower bounds are not within a constant factor. As a special case, let $\vN = (N, \sqrt{N}(\log N)^k)$ for $k \geq \frac{3}{2}$ and large $N$. Theorem~\ref{thm:rectangular} yields a lower bound of $\Omega(N^{\frac{1}{4}}(\log N)^{\frac{k}{6}})$ and an upper bound of $O(N^{\frac{1}{4}}{(\log N)^{\frac{k}{6}+1}}{(\log\log N)^{-1}})$. If we apply Lemma~\ref{lem:slice} 
and the Matou\v{s}ek-Spencer theorem in one dimension \cite{MS}, we get a weaker upper bound of $O(N^{\frac{1}{4}}(\log N)^{\frac{k+1}{2}})$. In some other regimes, such as when $0 < k<\frac{3}{2}$ in the above example, Lemma~\ref{lem:slice} and \cite{MS} gives a better upper bound than Theorem~\ref{thm:rectangular}, yet it is still not within a constant factor from the lower bound.


It is interesting to know if the sub-logarithmic factor in the upper bound of Theorem~\ref{thm:rectangular} can be removed or not. We conjecture that it can be and the lower bound is tight. 
\begin{conjecture}
For any integer $d\geq 1$, let $\vN = (N_1, N_2, \cdots, N_d)$ where $N_1,\dots, N_d$ are positive integers. Then
\[\disc(\mathcal{A}_\vN) = \Theta_d\left(\max_{I\subseteq [d]}\left(\prod_{i\in I}N_i\right)^\frac{1}{2|I|+2}\right).\]
\end{conjecture}

\end{document}